\documentclass{amsart}
\usepackage{amsmath}%
\usepackage{amsthm}
\usepackage{amsfonts}%
\usepackage{amssymb}%
\usepackage{graphicx}
\usepackage{tikz-cd}
\usetikzlibrary{cd}

\usepackage{mathrsfs}

\DeclareFontFamily{OT1}{pzc}{}
\DeclareFontShape{OT1}{pzc}{m}{it}{<-> s * [1.10] pzcmi7t}{}
\DeclareMathAlphabet{\mathpzc}{OT1}{pzc}{m}{it}

\newtheorem{theorem}{Theorem}[section]

\newtheorem{corollary}[theorem]{Corollary}

\newtheorem{definition}[theorem]{Definition}

\newtheorem{lemma}[theorem]{Lemma}

\newtheorem{proposition}[theorem]{Proposition}
\newtheorem{remark}[theorem]{Remark}

\numberwithin{equation}{section}

\def\XXint#1#2#3{{\setbox0=\hbox{$#1{#2#3}{\int}$}
\vcenter{\hbox{$#2#3$}}\kern-.5\wd0}}

\newcommand{\R}{\mathbb{R}}

\newcommand{\N}{\mathbb{N}}
\newcommand{\Z}{\mathbb{Z}}

\newcommand{\G}{\mathscr G}

\newcommand{\spt}{\operatorname{spt}}

\newcommand{\Lip}{\operatorname{Lip}}
\newcommand{\LIP}{\operatorname{LIP}}

\newcommand{\Parcon}{\Gamma_{\operatorname{pc,c}}^k}
\newcommand{\hParcon}{\overline\Gamma_{\operatorname{pc,c}}^k}
\newcommand{\ud}{\mathrm {d}}
\newcommand{\id}{\mathrm {id}}

\newcommand{\inv}{^{-1}}
\newcommand{\Rep}{\mathrm{Rep}}
\newcommand{\Poly}{{\sf Poly}}
\newcommand{\Polys}{\mathpzc{Poly}}

\newcommand{\sign}{\mathrm{sign}}
\newcommand{\Alt}{\mathrm{Alt}}

\newcommand{\loc}{\mathrm{loc}}

\newcommand{\sD}{\mathscr D}
\newcommand{\sB}{\mathscr B}

\newcommand{\cD}{\mathcal D}

\newcommand{\cU}{\mathcal U}

\title[Metric currents and Polylipschitz forms]{Metric currents and Polylipschitz forms}

\thanks{P.P. was supported in part by the Academy of Finland project \#297258. }
\keywords{Metric currents, differential forms on metric spaces, polylipschitz sheaves}
\subjclass[2010]{49Q15, 53C23, 30L99}

\author{Pekka Pankka}
\address{P.O. Box 68 (Pietari Kalmin katu 5), FI-00014 University of Helsinki, Finland}
\email{pekka.pankka@helsinki.fi}

\author{Elefterios Soultanis}
\address{SISSA, Via Bonomea 265, 34136, Trieste, Italy, and 
University of Fribourg, Chemin du Musee 23, CH-1700, Fribourg, Switzerland}

\email{elefterios.soultanis@gmail.com}

\date{\today}

\begin{document}

\begin{abstract}
We construct, for a locally compact metric space $X$, a space of polylipschitz forms $\overline\Gamma^*_c(X)$, which is a pre-dual for the space of metric currents of $\sD_*(X)$ Ambrosio and Kirchheim. These polylipschitz forms may be seen as an analog of differential forms in the metric setting.
\end{abstract}

\maketitle


\section{Introduction}

In \cite{amb00}, Ambrosio and Kirchheim extended the Federer--Fleming theory of currents to general metric spaces by substituting the differential structure on the domain for carefully chosen conditions on the functionals: a metric $k$-current $T\in \sD_k(X)$ on a metric space $X$ is a $(k+1)$-linear map $T:\sD^k(X):=\LIP_c(X)\times\LIP_\infty(X)^k\to \R$ satisfying continuity and locality conditions.

In this article we construct a pre-dual for the space of metric currents. Our strategy is to pass from $(k+1)$-tuples of Lipschitz functions to linearized and localized objects we call \emph{polylipschitz forms}. Linearization of multilinear functionals naturally involves tensor products, and we use sheaf theoretic methods to carry out the localization. 

Williams \cite{wil12} and Schioppa \cite{sch16} have given different constructions for pre-duals of metric currents. Their constructions are based on representation of currents of finite mass by duality using Cheeger differentiation and Alberti representations, respectively. Our motivation to consider polylipschitz forms stems from an application of metric currents to geometric mapping theory -- polylipschitz forms induce a natural local pull-back for metric currents of finite mass under BLD-mappings. We discuss this application briefly in the end of the introduction and in more detail in \cite{teripekka2}.

\subsection*{Polylipschitz forms and sections}
Polylipschitz forms are introduced in three steps: first polylipschitz functions and polylipschitz sections, then homogeneous polylipschitz functions, and finally polylipschitz forms. Before stating our results, we discuss the motivation for this hierarchy of spaces. 

The space $\Poly^k(X)$ of \emph{$k$-polylipschitz functions} on $X$ is the projective tensor product of $(k+1)$ copies of $\LIP_\infty(X)$. The collection $\{ \Poly^k(U)\}_U$, where $U$ ranges over open sets in $X$, forms a presheaf and gives rise to the \'etal\'e space of germs of polylipschitz functions. We denote $\Gamma^k(X)$ the space of continuous sections over this \'etal\'e space and call its elements polylipschitz sections. 

Although metric $k$-currents on $X$ act naturally on compactly supported polylipschitz sections $\Gamma^k_c(X)$ (see Theorem \ref{thm:polysec}), these sections are not the natural counterpart for differential forms on $X$, since the dual $\Gamma_c^k(X)^*$ contains functionals, which do not satisfy the locality condition for metric currents. Recall that in Euclidean spaces the $(k+1)$-tuple $(\pi_0,\ldots,\pi_k)\in\sD^k(X)$ corresponds to the measurable differential form $\pi_0 d\pi_1\wedge\dots \wedge d\pi_k$. The locality condition of Ambrosio and Kirchheim for metric currents states that $T(\pi_0,\ldots,\pi_k)=0$ if $\pi_l$ is constant on $\spt\pi_0$ for some $l>0$, while the polylipschitz section corresponding to $(\pi_0,\ldots,\pi_k)$ need not be zero, cf.~\eqref{eq:iota}.

For this reason, we introduce \emph{homogeneous polylipschitz functions}. These are elements of the projective tensor product $\overline\Poly^k(X)$ of $\LIP_\infty(X)$ and $k$ copies of $\overline\LIP_\infty(X)$, the space of bounded Lipschitz functions modulo constants. A \emph{polylipschitz form} is a continuous section over the \'etal\'e space $\overline\Polys^k(X)$ associated to the presheaf $\{\overline\Poly^k(U)\}_U$ and we denote the space of polylipschitz forms by $\overline\Gamma^k(X)$. The locality property of functionals in $\overline\Gamma^k(X)^*$, which motivated the homogeneous spaces, is discussed in Lemma 7.5.

We consider polylipschitz forms as differential forms in the metric setting, although we do not impose antisymmetry on polylipschitz forms. Note that, by an observation of Ambrosio and Kirchheim, the other properties imply the corresponding antisymmetry property for metric currents. In Section \ref{sec:Alt} we discuss how antisymmetry of polylipschitz forms can imposed a posteriori.

The space $\overline\Gamma_c^k(X)$ of compactly supported polylipschitz forms may be equipped with a notion of sequential convergence. There is a natural, sequentially continuous exterior derivative $\bar d:\overline\Gamma^k_c(X)\to \overline\Gamma^{k+1}_c(X)$, a pointwise norm $\|\cdot\|_x$, for $x\in X$, corresponding to the comass of a differential form, and a natural, sequentially continuous map
\[
\bar\iota:\sD^k(X)\to \overline\Gamma^k_c(X),
\]
cf.~\eqref{eq:iota}.

Our first main result states that the space of metric $k$-currents $\sD_k(X)$ on $X$ embeds bijectively into the sequentially continuous dual $\overline\Gamma_c^k(X)^\ast$ of $\overline\Gamma_c^k(X)$.

\begin{theorem}
	\label{thm:pre_dual}
	Let $X$ be a locally compact metric space and $k\in\N$. For each $T\in \sD_k(X)$ there exists a unique $\widehat T\in \overline\Gamma_c^k(X)^*$ for which the diagram \begin{equation}
	\begin{tikzcd}
	\sD^k(X)\arrow[swap]{d}{\bar\iota} \arrow[r, "T"] & \R\\
	\overline\Gamma_c^k(X)\arrow[swap]{ur}{\widehat T}
	\end{tikzcd}
	\end{equation}
	commutes. The map $T\mapsto \widehat{T}:\sD_k(X)\to \overline\Gamma_c^k(X)^\ast$ is a bijective and sequentially continuous linear map.
	
	Moreover, for each $T\in \sD_k(X)$ and  $\omega\in \overline\Gamma_c^{k-1}(X)$, we have 
\begin{equation}\label{main1:partial}
\widehat{\partial T}(\omega)=\widehat T(\bar d\omega).
\end{equation}
\end{theorem}

The sequential continuity of the map $T\mapsto \widehat T$ is defined as follows: \emph{Suppose that a sequence $(T_i)$ in $\sD_k(X)$ weakly converges to $T\in\sD_k(X)$, that is, for each $(\pi_0,\ldots,\pi_k)\in \sD^k(X)$, we have that $\lim_{i \to \infty} T_i(\pi_0,\ldots,\pi_k) =  T(\pi_0,\ldots,\pi_k)$. Then, for each $\omega \in \overline\Gamma_c^k(X)$, we have that $\lim_{i \to \infty} \widehat T_i(\omega) = \widehat T(\omega)$.}

\begin{remark}
A version of Theorem \ref{thm:pre_dual} for polylipschitz sections, shows that there is a natural sequentially continuous embedding $\sD_k(X)\to \Gamma_c^k(X)^*$; see Theorem \ref{thm:polysec}. As already discussed, this embedding is, however, not a surjection. 
\end{remark}

The space $\overline\Gamma_c^k(X)$ is a pre-dual to $\sD_k(X)$ in the sense of Theorem \ref{thm:pre_dual}. In the other direction, we remark that De Pauw, Hardt and Pfeffer consider in \cite{dep17} the dual of normal currents, whose elements are termed \emph{charges}. We do not consider charges here and merely note that the bidual $\overline\Gamma_c^k(X)^{\ast\ast}$ does not coincide with the space of charges.

\subsection*{Currents of locally finite mass and partition continuous polylipschitz forms}

The extension of a current of locally finite mass, provided by Theorem \ref{thm:pre_dual}, satisfies the following natural estimate.

\begin{theorem}\label{main2}
For each $T\in M_{k,\loc}(X)$ we have
	\begin{equation}\label{main1:mass}
	|\widehat{T}(\omega)|\le \int_X\|\omega\|\ud\|T\|
	\end{equation}
	for every $\omega\in\overline\Gamma_c^k(X)$.
\end{theorem}
We refer to Definition \ref{pwnorm} for the pointwise norm of a polylipschitz form. Currents of locally finite mass may further be extended, in the spirit of \cite[Theorem 4.4]{lan11}, to the space $\overline\Gamma_{\operatorname{pc},c}^k(X)$ of partition continuous polylipschitz forms, that is, the space of partition continuous sections of the sheaf $\overline\Polys^k(X)$; we refer to Section \ref{sec:aux} and Definition \ref{def:pcpolylip} for definitions and discussion.

\begin{theorem}\label{ext2}
	Let $T\in M_{k,\loc}(X)$ be a metric $k$-current of locally finite mass. Then there exists a unique sequentially continuous linear functional 
	\[ 
	\widehat T:\overline\Gamma_{\operatorname{pc},c}^k(X)\to \R 
	\] 
	satisfying $\widehat T\circ\iota=T$. Furthermore, if $T\in N_{k,\loc}(X)$, then
	\begin{equation}\label{eq:pcd}
	\widehat{\partial T}(\omega)=\widehat T(\bar d\omega)
	\end{equation}
	for each $\omega\in \overline\Gamma_{\operatorname{pc},c}^{k-1}(X)$.
\end{theorem}

Theorem \ref{ext2} follows directly from Proposition \ref{prop:pcext} and  Corollary \ref{pcbdry}, while Theorem \ref{main2} is implied by the more technical statement in Proposition \ref{prop:eext}. Note that, in Theorem \ref{ext2}, we do not claim that $\overline\Gamma_{\operatorname{pc},c}^k(X)$ is a pre-dual of $M_{k,loc}(X)$.

\subsection*{Motivation: Pull-back of metric currents by BLD-maps} In \cite{teripekka2} we apply the duality theory developed in this paper to a problem in geometric mapping theory. To avoid the added layer of abstraction involved in polylipschitz forms we formulate the results in \cite{teripekka2} for polylipschitz sections, which are sufficient for our purposes. For this reason, in Section \ref{sec:Alt} we briefly discuss duality theory in connection with polylipschitz sections.

In the Ambrosio--Kirchheim theory a Lipschitz map $f\colon X\to Y$ induces a natural push-forward $f_\# \colon M_{k}(X)\to M_{k}(Y)$. In the classical setting of Euclidean spaces, this push-forward is associated to the pull-back of differential forms under the mapping $f$. In \cite{teripekka2}, we consider BLD-mappings $f\colon X\to Y$ between metric generalized $n$-manifolds. A mapping $f\colon X\to Y$ is a mapping of bounded length distortion (or \emph{BLD} for short) if $f$ is a discrete and open mapping for which there exists a constant $L\ge 1$ satisfying
\[
\frac{1}{L} \ell(\gamma) \le \ell(f\circ \gamma) \le L \ell(\gamma)
\]
for all paths $\gamma$ in $X$, where $\ell(\cdot)$ is the length of a path. We refer to Martio--V\"ais\"al\"a \cite{mar88} and Heinonen--Rickman \cite{hei02} for detailed discussions on BLD-mappings between Euclidean and metric spaces, respectively.

For a BLD-mapping $f\colon X\to Y$ between locally compact spaces, the polylipschitz sections admit a push-forward $f_\# \colon \Gamma^k_c(X)\to \Parcon(Y)$, which in turn induces a natural pull-back $f^* \colon M_{k,\loc}(Y) \to M_{k,\loc}(X)$ for metric currents. We refer to \cite{teripekka2} for detailed statements and further applications.

\medskip\noindent {\bf Acknowledgments} We thank Rami Luisto and Stefan Wenger for discussions on the topics of the manuscript.


\section{Spaces of Lipschitz functions}
\label{sec:preliminaries}

We write $A\lesssim_{\mathcal C} B$ if there is a constant $c>0$ depending only on the parameters in the collection $\mathcal C$, for which $A\le c B$. We write $A\simeq_{\mathcal C}B$ if $A\lesssim_{\mathcal C} B\lesssim_{\mathcal C} A$.

Let $X$ be a metric space. We denote by $B_r(x)\subset X$ the open ball of radius $r>0$ about $x\in X$. The closed ball of radius $r>0$ about $x\in X$ is denoted by $\bar B_r(x)$.

\subsection{The spaces $\LIP_c$ and $\LIP_\infty$}
\label{sec:Lipschitz_constants} 

Given a Lipschitz map $f\colon X\to Y$ between metric spaces $(X,d)$ and $(Y,d')$, we denote by
\[ 
\Lip(f)=\sup_{x\ne y}\frac{d'(f(x),f(y))}{d(x,y)} 
\] 
the \emph{Lipschitz constant of $f$}. Further, for each $x\in X$, we denote
\begin{align*}
\Lip f(x)=\limsup_{r\to 0}\sup_{y,z\in B(x,r)}\frac{d'(f(y),f(z))}{d(z,y)}.
\end{align*}
the \emph{asymptotic} Lipschitz constant of $f$ at $x$.

For Lipschitz functions $f\colon X\to \R$, we introduce the norms
\[ 
\|f\|_\infty=\sup_{x\in X}|f(x)|\quad \text{and}\quad
L(f)=\max\{\Lip(f),\|f\|_\infty \}. 
\]

In what follows, we denote by $\LIP_\infty(X)$ the space of all \emph{bounded} Lipschitz  functions on $X$. Note that $(\LIP_\infty(X),L(\cdot))$ is a Banach space \cite{wea99}. Given a compact set $K\subset X$ we denote by $\LIP_K(X)$ the subspace of functions $f\in\LIP_\infty(X)$ satisfying $\spt\pi\subset K$.

The subspace $\LIP_c(X)\subset \LIP(X)$ consisting  of \emph{compactly supported} Lipschitz functions on $X$ is the union $$\LIP_c(X)=\bigcup\left\{\LIP_K(X): K\subset X\textrm{ compact}\right\}.$$ For each $k\in \N$, we also denote by $\sD^k(X)$ the product space 
\[
\sD^k(X)=\LIP_c(X)\times \LIP_\infty(X)^k.
\]

\subsection{Homogeneous Lipschitz space}\label{sec:hlip}

We introduce now the homogeneous Lipschitz space $\overline\LIP_\infty(X)$. The term homogeneous is taken from the theory of Sobolev spaces, where analogous homogeneous Sobolev spaces are defined. 

Let $\sim$ be the equivalence relation in $\LIP_\infty(X)$ for which $f \sim f'$ if $f-f'$ is a constant function. We denote the equivalence class of $f\in \LIP_\infty(X)$ by $\bar f$ and give the quotient space $\overline\LIP_\infty(X):=\LIP_\infty(X)/{\sim}$ the quotient norm
\[
\bar L(\bar f):=\inf\{ L(f-c):c\in \R\}.
\]
The natural projection map
\[
q=q_X:\LIP_\infty(X)\to \overline\LIP_\infty(X),\quad f\mapsto \bar f,
\]
is an open surjection satisfying
\[
\bar L(\bar f)=\inf\{ L(g): q(g)=\bar f \}.
\]

Note that, given a subset $E\subset X$, the restriction map
\[
r_{E,X}:\LIP_\infty(X)\to \LIP_\infty(E),\quad f\mapsto f|_E,
\]
descends to a quotient map 
\begin{equation}\label{prerestr}
\bar r_{E,X}:\overline\LIP_\infty(X)\to \overline\LIP_\infty(E)
\end{equation}
satisfying
\begin{equation}\label{eq:restrcomm}
q_E\circ r_{E,X}=\bar r_{E,X}\circ q_X.
\end{equation}
This remark will be used later in Section \ref{sec:germs}.

\subsection{Sequential convergence}

Following Lang \cite{lan11}, we give the spaces $\LIP_c(X)$ and $\LIP_\infty(X)$ with the topology of weak converge. We recall the notion of convergence of sequences in $\LIP_c(X)$ and refer to \cite{lan11} for the definition of the corresponding topology; see also Ambrosio--Kirchheim \cite{amb00}.

\begin{definition}
A sequence $(f_n)$ in $\LIP_c(X)$ \emph{converges weakly to a function $f\colon X\to \R$ in $\LIP_c(X)$}, denoted $f_n \to f$ in $\LIP_c(X)$, if 
\begin{enumerate}
\item $\sup_n\Lip(f_n)<\infty$, 
\item the set $\bigcup_n \spt f_n$ is pre-compact, and
\item $f_n \to f$ uniformly as $n \to \infty$.
\end{enumerate}
\end{definition}

In \cite{lan11} Lang defines a topology on a larger space $\LIP_\loc(X)$, the space of \emph{locally Lipschitz} functions, containing $\LIP_\infty(X)$. The weak convergence induced by this topology for sequences could be used for functions in $\LIP_\infty(X)$ as well. However $\LIP_\infty(X)\subset\LIP_\loc(X)$ is not a closed subspace in this topology and we find it is more convenient to modify the notion of convergence to suit the space $\LIP_\infty(X)$ better. This does not cause any significant issues in the subsequent discussions. The weak convergence of sequences in $\LIP_\infty(X)$ is defined as follows.

\begin{definition}\label{multicont}
A sequence $(f_n)$ in $\LIP_\infty(X)$ \emph{converges to a function $f \colon X\to \R$ in $\LIP_\infty(X)$}, denoted $f_n\to f$ in $\LIP_\infty(X)$, if 
\begin{enumerate}
\item $\sup_n\Lip(f_n)<\infty$, and 
\item $f_n|_K\to f|_K$ uniformly as $n\to \infty$, for every compact set $K\subset X$.
\end{enumerate}
\end{definition}
This notion of convergence for sequences arises from a topology in a similar manner as in \cite{lan11}. Another description of this topology is given in \cite[Theorems 2.1.5 and 1.7.2]{wea99} in terms of the weak* topology with respect to the Arens--Eells space, which is a predual of $\LIP_\infty(X)$.

We equip the product space $\sD^k(X)=\LIP_c(X)\times \LIP_\infty(X)^k$ with the sequential convergence arising from the product topology of the factors $\LIP_c(X)$ and $\LIP_\infty(X)$.

A sequence $(\pi^n)$ of $(k+1)$-tuples $\pi^n=(\pi_0^n,\ldots,\pi^n_k)\in \sD^k(X)$ converges to a $(k+1)$-tuple $\pi=(\pi_0,\ldots,\pi_k)$ in $\sD^k(X)$ if and only if 
\begin{enumerate}
\item $\pi_0^n\to \pi_0$ in $\LIP_c(X)$ and
\item $\pi^n_l\to \pi_l$ in $\LIP_\infty(X)$, for each $l=1,\ldots,k$
\end{enumerate}
as $n\to \infty$.

We finish this section by defining the sequential continuity of multilinear functionals on $\sD^k(X)$.

\begin{definition} 
A multilinear functional $T\colon \sD^k(X) \to \R$ is \emph{sequentially continuous} if, for each sequence $(\pi^n)$ converging to $\pi$ in $\sD^k(X)$, 
\[
\lim_{n\to\infty}T(\pi^n)= T(\pi).
\]
\end{definition}

\section{Metric currents} 
\label{label:metric_currents}

Let $X$ be a locally compact metric space. A sequentially continuous multilinear functional $T\colon \sD^k(X) \to \R$  is a \emph{metric $k$-current} if it satisfies the following locality condition:
\begin{itemize}
\item[] for each $\pi_0\in \LIP_c(X)$ and $\pi_1,\ldots, \pi_k\in \LIP_\infty(X)$ having the property that one of the functions $\pi_i$ is constant in a neighborhood of $\spt \pi_0$, we have $T(\pi_0,\ldots, \pi_k) = 0$.
\end{itemize}

By \cite[(2.5)]{lan11} the assertion in the locality condition holds if one of the $\pi_i$'s, $i=1,\ldots,k$, is constant on $\spt\pi_0$, i.e. no neighborhood is needed in the locality condition. We denote by $\sD_k(X)$ the vector space of $k$-currents on $X$. 

\begin{remark}
In \cite{lan11} metric currents are defined as weakly continuous $(k+1)$-linear functionals on $\LIP_c(X)\times\LIP_\loc(X)^k$. The present notion however coincides with this class, see \cite[Lemma 2.2]{lan11}.
\end{remark}

\begin{definition}
A sequence of $k$-currents $T_j \colon \sD^k(X)\to \R$ on $X$ \emph{converges weakly to a $k$-current $T \colon \sD^k(X) \to \R$ in $\sD_k(X)$}, denoted $T_j\to T$ in $\cD^k(X)$, if, for each $\pi\in \sD^k(X)$, 
\[
\lim_{j\to\infty}T_j(\pi)= T(\pi).
\]
\end{definition}

The locality condition implies that the value $T(\pi_0,\ldots,\pi_k)$ of a current $T\in \sD_k(X)$ at $(\pi_0,\ldots,\pi_k)\in \sD^k(X)$ only depends on the restriction $\pi_0|_K\in \LIP_\infty(K)$ and the equivalence classes of the restrictions  $\overline{\pi_1|_K},\ldots,\overline{\pi_k|_K}\in\overline\LIP_\infty(K)$, where $K=\spt\pi_0$.

Given a compact set $K\subset X$, we define $T: \LIP_K(X)\times\LIP_\infty(K)^k\to \R$ as 
\[ 
T(\pi_0,\ldots,\pi_k):=T(\pi_0,\widetilde{\pi_1},\ldots,\widetilde{\pi_k}) 
\]
for any Lipschitz extensions $\widetilde{\pi_i}\in \LIP_\infty(X)$ of $\pi_i\in \LIP_\infty(K)$, for $i=1,\ldots,k$. This yields a $(k+1)$-linear sequentially continuous functional. More precisely, there exists $C>0$ for which 
\begin{equation}
\label{eq:bound}
|T(\pi_0,\ldots,\pi_k)|\le CL(\pi_0|_K)L(\pi_1)\cdots L(\pi_k) 
\end{equation}
for all $\pi_0\in \LIP_{K}(X)$, $\pi_1,\ldots,\pi_k\in \LIP_\infty(K)$.

\subsection{Mass of a current}
A $k$-current $T \colon \sD^k(X) \to \R$ has \emph{locally finite mass} if there is a Radon measure $\mu$ on $X$ satisfying
\begin{equation}\label{locfin}
|T(\pi_0,\ldots,\pi_k)|\le \Lip(\pi_1)\cdots\Lip(\pi_k)\int_X|\pi_0|\ud\mu
\end{equation}
for each $\pi = (\pi_0,\ldots, \pi_k) \in \sD^k(X)$.

For a $k$-current $T$ of locally finite mass, there exists a measure $\|T\|$ of minimal total variation satisfying \eqref{locfin}; see Lang \cite[Theorem 4.3]{lan11}. The measure $\|T\|$ is the \emph{mass measure of $T$}. If $\|T\|(X)<\infty$, we say $T$ has \emph{finite mass}. We denote by $M_{k,\loc}(X)$ and $M_k(X)$ the spaces of $k$-currents of locally finite mass and of finite mass, respectively.

The map $\sD^k(X)\to \mathscr M(X)$, $T\mapsto \|T\|$, is lower semicontinuous with respect to weak convergence, that is, if the sequence $(T_j)$ in $M_{k,\loc}(X)$ weakly converges to $T\in M_{k,\loc}(X)$ then 
\[ 
\|T\|(U)\le \liminf_{j\to\infty}\|T_j\|(U)
\]
for each open set $U\subset X$. Note that, 
\[
\spt T=\spt\|T\|.
\]
We refer to Lang \cite{lan11} for these results.

A current $T \colon \sD^k(X)\to \R$ of locally finite mass admits a weakly continuous extension 
\[ 
T:\sB^\infty_c(X)\times\LIP_\infty(X)^k\to \R 
\] 
satisfying (\ref{locfin}); see Lang \cite[Theorem 4.4]{lan11}. Here $\sB_c^\infty(X)$ denotes the space of compactly supported and bounded Borel functions on $X$; here the convergence of functions in $\sB^\infty_c(X)$ is the  pointwise convergence. Inequality \eqref{locfin} holds also for this extension. We record this as a lemma.

\begin{lemma}[{\cite[Theorem 4.4]{lan11}}] 
\label{moi}
Let $T\in M_{k,\loc}(X)$ be a $k$-current of locally finite mass. Then
\[ 
|T(\pi_0,\ldots,\pi_k)|\le \Lip(\pi_1|_E)\cdots \Lip(\pi_k|_E)\int_E|\pi_0|\ud\|T\| 
\] 
for each $(k+1)$-tuple $(\pi_0,\ldots,\pi_k)\in \sB^\infty_c(X)\times\LIP(X)^k$ and Borel set $E\supset \{\pi_0\ne 0 \}$ 
\end{lemma}

\subsection{Normal currents}
 
For the definition of a normal current, we first define the boundary operators 
\[
\partial = \partial_k  \colon \sD_k(X)\to \sD_{k-1}(X)
\]
for each $k\in \Z$. For this reason, we set $\sD_k(X)=0$ for $k<0$ and $\partial_k = 0$ for $k\le 0$.

For $k\ge 1$, the \emph{boundary $\partial T\colon \sD^{k-1}(X)\to \R$ of $T$} is the $(k-1)$-current defined by
\[
\partial T( \pi_0, \ldots, \pi_{k-1}) = T(f, \pi_0, \ldots, \pi_{k-1}),
\]
for $(\pi_0,\ldots,\pi_k)\in \sD^k(X)$, where $f\in \LIP_c(X)$ is any Lipschitz function  with compact support satisfying
\[
f|_{\spt \pi_0} \equiv 1.
\] 
The current $\partial T$ is well-defined, see \cite[Definition 3.4]{lan11}.

As a consequence of the locality of currents, we have that 
\[
\partial_{k-1} \circ \partial_k \equiv 0 \colon \sD_k(X) \to \sD_{k-2}(X)
\] 
for each $k\in \Z$, see the discussion following \cite[Definition 3.4]{lan11}.

\begin{definition}
A $k$-current $T\in M_{k,\loc}(X)$ is \emph{locally normal} if $\partial T\in M_{k-1,\loc}(X)$. A $k$-current $T\in M_{k}(X)$ is \emph{normal} if $\partial T\in M_{k-1}(X)$. 
\end{definition}

We denote by $N_{k,\loc}(X) \subset M_{k,\loc}(X)$ and $N_k(X) \subset M_k(X)$ the subspaces of locally normal $k$-currents and normal $k$-currents on $X$, respectively. Note that the space $N_0(X)$ of normal 0-currents coincides with the space $\mathscr M(X)$ of all finite signed Radon measures on $X$.

\begin{remark}
By the lower semicontinuity of mass, if a bounded sequence $(T_j)$ in $N_k(X)$ weakly converges to a $k$-current $T\in M_{k,\loc}(X)$, then $T\in N_k(X)$. 
\end{remark}

\section{Polylipschitz functions and their homogeneous counterparts}
\label{sec:polylipschitz_forms}

In this section we develop the notion polylipschitz functions and their homogeneous counterparts. In the sequel we occasionally refer to (homogeneous) polylipschitz forms and functions defined on Borel subsets of a locally compact space. Since these subsets are not necessarily locally compact, and since the treatment remains essentially the same, we formulate all notions in this section for arbitrary metric spaces.

\subsection*{Algebraic and projective tensor product} 
For the material on the tensor product, projective norm and projective tensor product, we refer to \cite[Sections 1 and 2]{rya02}. 

Let $(V_1,\|\cdot\|_1),\ldots, (V_k,\|\cdot\|_k)$ be Banach spaces. We denote the \emph{algebraic tensor product} of $V_1,\ldots,V_k$ by $V_1\otimes \cdots \otimes V_k$. There is a natural $k$-linear map
\[
\jmath:V_1\times\cdots\times V_k\to V_1\otimes\cdots\otimes V_k,\quad (v_1,\ldots,v_k)\to v_1\otimes\cdots\otimes v_k.
\]
The \emph{projective norm} of $v\in V_1\otimes\cdots\otimes V_k$ is
\begin{equation}\label{eq:projnorm}
\pi(v)=\inf\left\{\sum_j^n \|v_1^j\|_1\cdots \|v_k^j\|_k: v=\sum_j^nv_1^j\otimes\cdots\otimes v_k^j  \right\}.
\end{equation}
The projective norm is \emph{a cross norm} \cite[Proposition 2.1]{rya02}, that is, for $v_l\in V_l$, $l=1,\ldots,k$, we have
\begin{equation}\label{eq:projnorm2}
\pi(v_1\otimes\cdots\otimes v_k)=\|v_1\|_1\cdots \|v_k\|_k.
\end{equation}
Note that by (\ref{eq:projnorm2}) the canonical $k$-linear map $\jmath$ is continuous; see \cite[Theorem 2.9]{rya02}.

The normed vector space $( V_1\otimes\cdots\otimes V_k,\pi)$ is typically not complete. Its completion $$ V_1\hat\otimes_\pi\cdots\hat\otimes_\pi V_k$$ is called the \emph{projective tensor product}. We denote by $\widehat \pi:V_1\hat\otimes_\pi\cdots\hat\otimes_\pi V_k\to [0,\infty)$ the norm on the completion $V_1\hat\otimes_\pi\cdots\hat\otimes_\pi V_k$ of $V_1\otimes\cdots\otimes V_k$.

It should be noted that the projective norm is one of many possible norms on the algebraic tensor product, each giving rise to a completion. In general these completions are not isomorphic and there is no canonical completion. However the projective tensor product has the following \emph{universal property} which characterizes it up to isometric isomorphism in the category of Banach spaces: Let $B$ be a Banach space and $$A:V_1\times\cdots\times V_k\to B$$ a continuous $k$-linear map. Then there exists a unique \emph{continuous linear map} $$ \overline A:V_1\hat\otimes_\pi\cdots\hat\otimes_\pi V_k\to B$$ for which the diagram
\begin{equation}\label{diagram:proj}
\begin{tikzcd}
V_1\times\cdots\times V_k\arrow[swap]{d}{\jmath} \arrow[r, "A"] & B\\
V_1\hat\otimes_\pi\cdots\hat\otimes_\pi V_k\arrow[swap]{ur}{\overline A}
\end{tikzcd}
\end{equation}
commutes.

Heuristically, the elements of $V_1\hat\otimes_\pi\cdots\hat\otimes_\pi V_k$ can be viewed as series or as summable sequences. More precisely, we have the following result.

\begin{theorem}\cite[Proposition 2.8]{rya02}
	Let $V_1,\ldots,V_k$ be Banach spaces, and let $v\in V_1\hat\otimes_\pi\cdots\hat\otimes_\pi V_k$. Then there is a sequence $(v_1^j,\ldots,v_k^j)_j$ in $V_1\times\cdots\times V_k$ for which 
	\begin{equation}\label{tensor:poly}
	\sum_j^\infty \|v_1^j\|_1\cdots \|v_k^j\|_k<\infty
	\end{equation}
	and
	\begin{equation}\label{tensor:polylip}
	\lim_{n\to\infty}\widehat \pi(v-\sum_j^nv_1^j\otimes\cdots\otimes v_k^j)=0.
	\end{equation}
\end{theorem}

\subsection{Polylipschitz functions}\label{sec:polylipfns}

We define in this section polylipschitz functions and consider their representations. The counterpart of this discussion for homogeneous polylipschitz functions is postponed to Section \ref{sec:hpoly}.

\begin{definition}\label{defn:polylipfns}
Let $X$ be a metric space and $k\in\N$. A \emph{$k$-polylipschitz function} on $X$ is an element in the $(k+1)$-fold projective tensor product $$\displaystyle \Poly^k(X):=\LIP_\infty(X)\hat\otimes_\pi\cdots\hat\otimes_\pi\LIP_\infty(X).$$
We denote by $L_k(\cdot)$ the projective tensor norm on $\Poly^k(X)$.
\end{definition}

Given $\pi_0,\ldots\pi_k\in\LIP_\infty(X)$ the tensor product $\pi_1\otimes\cdots\otimes\pi_k\in\Poly^k(X)$ may be identified with the function in $\LIP_\infty(X^{k+1})$ given by
\[
(x_0,\ldots,x_k)\mapsto \pi_0(x_0)\cdots\pi_k(x_k).
\]
Indeed, if $\Phi_X:\LIP_\infty(X)^{k+1}\to \LIP_\infty(X^{k+1})$ is the continuous $(k+1)$-linear map given by
\[
\Phi_X(\pi_0,\ldots,\pi_k)(x_0,\ldots,x_k):=\pi_0(x_0)\cdots\pi_k(x_k), \quad (x_0,\ldots,x_k)\in X^{k+1},
\]
the algebraic tensor product $\LIP_\infty(X)\otimes\cdots\otimes\LIP_\infty(X)$ may be identified with the linear span of $\Phi_X(\LIP_\infty(X)^{k+1})$ and the unique continuous linear map $\overline\Phi_X:\Poly^k(X)\to\LIP_\infty(X^{k+1})$ making the diagram (\ref{diagram:proj}) commute is injective.

Thus, we may regard a polylipschitz function $\pi\in \Poly^k(X)$ as a function $\pi:X^{k+1}\to \R$ for which there exists a sequence $(\pi_0^j,\ldots,\pi_k^j)_j$ in $\LIP_\infty(X)^{k+1}$ satisfying 
\begin{equation}\label{poly}
\sum_j^\infty L(\pi_0^j)\cdots L(\pi_k^j)<\infty
\end{equation}
and
\begin{equation}\label{polylip}
\pi=\sum_j^\infty \pi_0^j\otimes\cdots\otimes\pi_k^j
\end{equation}
pointwise. That is, we may identify $\Poly^k(X)$ with $\overline\Phi_X(\Poly^k(X))\subset \LIP_\infty(X^{k+1})$ as sets.

\begin{definition}\label{thm:represent}
	For $\pi\in\Poly^k(X)$, any sequence $(\pi_0^j,\ldots\pi_k^j)_j$ in $\LIP_\infty(X)^{k+1}$ satisfying (\ref{poly}) and (\ref{polylip}) -- or, equivalently (\ref{tensor:poly}) and (\ref{tensor:polylip}) --  is said to \emph{represent $\pi$.} We denote the collection of such sequences by $\Rep(\pi)$.
\end{definition}
Conversely, if a sequence $(\pi_0^j,\ldots,\pi_k^j)$ in $\LIP_\infty(X)^{k+1}$ satisfies (\ref{poly}) it represents a polylipschitz function. 

We denote by
\begin{equation}\label{eq:jmath}
\jmath=\jmath_X^k:\LIP_\infty(X)^{k+1}\to\Poly^k(X),\quad (\pi_0,\ldots,\pi_k)\mapsto \pi_0\otimes\cdots\otimes\pi_k
\end{equation}
the natural $(k+1)$-linear bounded map, cf. (\ref{diagram:proj}).

For metric spaces the standard McShane extension for Lipschitz functions yields immediately an extension also for polylipschitz functions. We record this as a lemma.

\begin{lemma}
	\label{ext}
	Let $X$ be a metric space, $E\subset X$ a subset, and let $\pi\in \Poly^k(E)$. Then there exists a $k$-polylipschitz function $\tilde{\pi}\in\Poly^k(X)$ extending $\pi$ and satisfying $L_k(\tilde{\pi}) = L_k(\pi)$.
	
	More precisely, if $(\pi^j_0,\ldots, \pi^j_k)_{j}$ is a representation of $\pi$ and $\tilde{\pi}_l^j\in \LIP_\infty(X)$ is an extension of $\pi_l^j$ satisfying $L(\tilde{\pi}^j_l) = L(\pi^j_i)$ for each $j\in \N$ and $l=0,\ldots, k$, then the sequence  $(\tilde{\pi}^j_0,\ldots, \tilde{\pi}^j_k)_{j}\subset \LIP_\infty(X)^{k+1}$ represents a polylipschitz function $\tilde\pi\in \Poly^k(X)$ for which $L_k(\tilde\pi)=L_k(\pi)$.
\end{lemma}

Thus polylipschitz functions defined on a subset $E\subset X$ can always be extended to polylipschitz functions on $X$ preserving the polylipschitz norm.

\subsection{Homogeneous polylipschitz functions}\label{sec:hpoly}
\begin{definition}
A \emph{homogeneous $k$-polylipschitz function} is an element in $$\displaystyle \overline\Poly^k(X):=\LIP_\infty(X)\hat\otimes_\pi\overline \LIP_\infty(X)\hat\otimes_\pi\cdots\hat\otimes_\pi\overline \LIP_\infty(X),$$ where $\overline\LIP_\infty(X)$ appears $k$ times in the tensor product. We denote by $\overline L_k(\cdot)$ the projective tensor norm on $\overline\Poly^k(X)$.
\end{definition}

Denote by 
\begin{equation}\label{eq:barjmath}
\bar\jmath=\bar\jmath_X^k:\LIP_\infty(X)^{k+1}\to \overline\Poly^k(X),\quad \bar\jmath:=\jmath\circ Q
\end{equation}
the natural bounded $(k+1)$-linear map.

The natural quotient map $q:\LIP_\infty(X)\to \overline\LIP_\infty(X)$ induces a quotient map
\begin{equation}\label{eq:quotient}
Q=Q_X^k:=\id\otimes_\pi q\otimes_\pi\cdots\otimes_\pi q:\Poly^k(X)\to \overline\Poly^k(X),
\end{equation}
that is, $\overline q:\Poly^k(X)\to\overline\Poly^k(X)$ is an open surjection and
\[
\bar L_k(\bar\pi)=\inf\{ L_k(\pi):\ Q(\pi)=\bar\pi \}\quad \text{for each }\bar\pi\in\overline\Poly^k(X),
\]
cf. \cite[Proposition 2.5]{rya02}. 

\subsection{Convergence of polylipschitz and homogeneous polylipschitz functions} 

In this subsection, we define a notion of convergence of sequences on $\Poly^k(X)$ and $\overline\Poly^k(X)$. These notions correspond to the weak$-^*$ convergence in $\LIP_\infty(X)$; see Section 2.1. We give the necessary notions of convergence in two separate definitions. 

\begin{definition}\label{conv}
A sequence $(\pi^n)$ in $ \Poly^k(X)$ converges to $\pi\in\Poly^k(X)$ if, for all compact sets $V:=V_0\times \cdots\times V_k\subset X^{k+1}$ there exists representations 
	\[
	(\pi_0^{n,j},\ldots,\pi_k^{j,n})\in \Rep(\pi^n-\pi)
	\]
	for which
	\begin{itemize}
		\item[(1)] $\displaystyle \sum_j^\infty\sup_nL(\pi_0^{j,n})\cdots L(\pi_k^{j,n})<\infty$, and
		\item[(2)] $\displaystyle \lim_{n\to\infty}\sum_j^\infty\|\pi_0^{j,n}|_{V_0}\|_\infty\cdots \|\pi_k^{j,n}|_{V_k}\|_\infty= 0.$
	\end{itemize}
\end{definition}

\begin{definition}\label{hconv}
A sequence $(\bar\pi^n)$ in $\overline\Poly^k(X)$  converges to $\bar\pi\in\overline\Poly^k(X)$ if there are polylipschitz functions $\pi\in Q\inv(\bar\pi)$ and $\pi^n\in Q\inv(\bar\pi^n)$ for each $n\in\N$, so that $\pi^n\to\pi$ in $\Poly^k(X)$.
\end{definition}
\begin{remark}
It follows immediately from Definitions \ref{conv} and \ref{hconv} that the natural maps $\jmath:\LIP_\infty(X)^{k+1}\to \Poly^k(X)$ and $\bar\jmath:\LIP_\infty(X)^{k+1}\to \overline\Poly^k(X)$ in (\ref{eq:jmath}) and (\ref{eq:barjmath}) are sequentially continuous.
\end{remark}

Before moving to polylipschitz forms, we record a notion of locality for linear maps in $\LIP_\infty(X)^{k+1}$ and record some of its consequences. 

\begin{definition}
\label{def:locality}
A $(k+1)$-linear map $A:\LIP_\infty(X)^{k+1}\to V$ is \emph{local} if, for a $(k+1)$-tuple $(\pi_0,\ldots, \pi_k)\in \LIP_\infty(X)^{k+1}$, holds  
\begin{equation}\label{eq:locality}
A(\pi_0,\ldots,,\pi_k)=0
\end{equation}
whenever one of the functions $\pi_1,\ldots, \pi_k$ is constant.
\end{definition}
 
\begin{proposition}\label{prop:useful}
Let $V$ be a Banach space and let $A:\LIP_\infty(X)^{k+1}\to V$ be a bounded and local $(k+1)$-linear map. Then 
\begin{enumerate}
\item $A$ descends to a (unique) bounded $(k+1)$-linear map $A':\LIP_\infty(X)\times\overline\LIP_\infty(X)^k\to V$,
\item the unique bounded linear maps $\overline A:\Poly^k(X)\to V$ and $\overline{A'}:\overline\Poly^k(X)\to V$ satisfying $A = \overline A \circ \jmath$ and $A' = \overline{A'} \circ (Q\circ \jmath)$, respectively, satisfy $\overline{A} = \overline{A'}\circ Q$, and
\item if $\overline A$ is sequentially continuous (in the sense of Definition \ref{conv}) then $\overline{A'}$ sequentially continuous (in the sense of Definition \ref{hconv}).
\end{enumerate}
\end{proposition}

\begin{proof}
It is clear that, since $A$ is local, it descends to a unique bounded multilinear map $A':\LIP_\infty(X)\times\overline\LIP_\infty(X)^k\to V$ satisfying $A=A'\circ(\id\times q_X\times\cdots\times q_X)$. Note that
\[
Q=\overline{\id\times q_X\times\cdots\times q_X}.
\]
It follows from the uniqueness of the diagram (\ref{diagram:proj}) that
\[
\overline A=\overline{A'\circ(\id\times q_X\times\cdots\times q_X)}=\overline{A'}\circ Q.
\]

Suppose $\overline{A}$ is sequentially continuous and let $\bar\pi^n\to 0$ in $\overline\Poly^k(X)$. For each $n\in\N$, we fix $\pi^n\in\Poly^k(X)$ so that $Q(\pi^n)=\bar\pi^n$ and that the sequence $(\pi^n)$ converges to $0\in Q^{-1}(0)$ in $\Poly^k(X)$. Then $\overline{A}(\pi^n)\in Q\inv(\overline{A'}(\bar\pi^n))$ and $\overline{A}(\pi^n)\to 0$. Thus $\overline{A'}(\bar\pi^n)\to 0$ in $\overline\Poly^k(X)$. Since $\overline{A'}$ is linear this proves the sequential continuity of $\overline{A'}$.
\end{proof}

\section{Polylipschitz forms and sections}\label{sec:polyform}
Since we consider presheaves of polylipschitz functions and homogeneous polylipschitz functions and their \'etal\'e spaces, we discuss the related terminology first in more general. We refer to \cite[Section 5.6]{war83} and \cite[Chapter II]{wel80} for a more detailed discussion.

\subsection{Presheaves and \'etal\'e spaces} 

Let $X$ be a paracompact Hausdorff space. A \emph{presheaf} $P$ on $X$ is a collection $\{A(U)\}_U$ of vector spaces (over $\R$) for each open set $U\subset X$ and, for each inclusion $U\subset V$, a linear map $\rho_{U,V}:A(V)\to A(U)$ satisfying $\rho_{U,U}=\id$ and
\begin{equation}\label{eq:restrcommute}
\rho_{U,V}=\rho_{U,W}\circ\rho_{W,V}
\end{equation}
whenever $U\subset W\subset V$.

Given two presheaves $\{A(U) \}$ and $\{B(U)\}$ on $X$, a collection $$\{ \varphi_U:B(U)\to A(U) \}$$ of linear maps satisfying 
\begin{equation}\label{eq:idhomom}
\varphi_U\circ\rho^B_{U,V}=\rho_{U,V}^A\circ\varphi_V,\quad U\subset V
\end{equation}
is called a \emph{presheaf homomorphism}.

Given an open set $U\subset X$, the \emph{support} of $f\in A(U)$, denoted $\spt(f)$, is the intersection of all closed sets $F\subset U$ with the property that $\displaystyle \rho_{U\setminus F,U}(f)=0$.

\subsubsection*{Fine presheaves} A presheaf $\{ A(U) \}$ on $X$ is called \emph{fine} if every open cover of $X$ admits a locally finite open refinement $\cU$ and, for each $U\in \cU$, there is a presheaf homomorphism 
\[
\{ (L_U)_W:A(W)\to A(W) \}
\]
with the following properties:
\begin{itemize}
	\item[(a)] $\spt (L_U)_W(f)\subset U\cap W$ for every $f\in A(W)$ and $U\in\cU$,
	\item[(b)] every point $x\in X$ has a neighborhood $D\subset X$ for which $D\cap U\ne\varnothing$ for only finitely many $U\in\cU$ and 
	$$\sum_{U\in\cU}\rho_{D,W}\circ (L_U)_W=\rho_{D,W}.$$
	whenever $D\subset W$.
\end{itemize}
Note that by (a) and the assumption on $D$, the sum in (b) has only finitely many non-zero terms.

\subsubsection*{Space of germs and its sections} 
Let $x\in X$ and $U$  and $V$ be open neighborhoods of $x$. Two elements $f\in A(U)$ and $g\in A(V)$ are equivalent if there exists an open neighborhood $D\subset U\cap V$ of $x$ so that
\[
\rho_{D,U}(f)=\rho_{D,V}(g).
\]
This defines an equivalence relation on the disjoint union $\displaystyle \bigsqcup_UA(U)$. We denote by $\mathcal A(X)$ the set of equivalence classes and say that $\mathcal A(X)$ is the space of germs for the presheaf $P$.

Given $x\in X$, an open neighborhood $U$ of $x$ and $f\in A(U)$ we denote by $[f]_x\in \mathcal A(X)$ the equivalence class of $f$ and call it the \emph{germ of $f$ at $x$}. There is a natural projection map
\begin{equation}\label{eq:proj}
p:\mathcal A(X)\to X,\quad [f]_x\mapsto x
\end{equation}
and the fibers $p\inv(x)=:\mathcal A_x(X)$ are called \emph{stalks of $\mathcal A(X)$ over $x$.} The stalk $\mathcal A_x(X)$ has a natural addition and scalar multiplication, making it a vector space (see \cite[Section 5.6]{war83}).

If $U\subset X$ is an open set, a map $\omega:U\mapsto \mathcal A(U)$ satisfying
\[
p\circ \omega=\id
\]
is called a \emph{section of $\mathcal A(U)$ over $U$} and the space of all sections of $\mathcal A(U)$ over $U$ is denoted by $\G(U;\mathcal A(U)).$ Note that $\G(U;\mathcal A(U))$ has a natural vector space structure given by pointwise addition and scalar multiplication. We abbreviate $\G(\mathcal A(X))=\G(X;\mathcal A(X))$ and call elements of $\G(\mathcal A(X))$ \emph{global sections of $\mathcal A(X)$}.

\subsubsection*{\'Etal\'e space} There is a natural \emph{\'etal\'e topology} on $\mathcal A(X)$ so that the projection map (\ref{eq:proj})  is a local homeomorphism.

The \'etal\'e topology has a basis of open sets of the form
\[
O_{U,f}=\{ [f]_x: x\in U \}
\]
for $U\subset X$ open and $f\in A(U)$, cf. \cite[Section 5.6]{war83}. We call $\mathcal A(X)$ equipped with this topology the \emph{\'etal\'e space} associated to the presheaf $P$.

If $\omega\in\G(\mathcal A(X))$ and $\cU$ is an open cover of $X$, a collection $\{ f_U\in A(U) \}_\cU$ is called  \emph{compatible with $\omega$} if, for every $x\in X$, there exists $U\in\mathcal U$ so that $\omega(x)=[f_U]_x$. Note that 
\[
\omega\inv O_{U,f_U}=\{ x\in U: \omega(x)=[f_U]_x \}\subset U
\]
forms a cover of $X$ which is a refinement of $\cU$. When $\omega$ is continuous, the open cover $\mathcal V=\{ \omega\inv O_{U,f_U} \}_\cU$ and the collection $g_V:=\rho_{V,U}(f_U)$, where $V=\omega\inv O_{U,f_U}$, is compatible with $\omega$ and furthermore
\begin{equation}\label{eq:represent}
[g_V]_x=\omega(x)=[g_W]_x \textrm{ whenever $V,W\in \mathcal V$ and $x\in V\cap W$.}
\end{equation}
We say that a collection $\{ f_U \}_\cU$ \emph{represents} a continuous section $\omega\in\G(\mathcal A(X))$, if it satisfies  (\ref{eq:represent}).

Fine presheaves with a mild additional assumption admit a stronger form of (\ref{eq:represent}), called the \emph{overlap condition}, for collections representing continuous sections, which we record as the following lemma. This will be used in Section \ref{sec:currentdual} to define the action of a current on polylipschitz forms.

\begin{lemma}\label{lem:overlap}
    Let $\{A(U)\}$ is a fine presheaf and suppose that the linear maps $\rho_{U,X}$ are onto for each open $U\subset X$.
    
	If $\omega\in\G(\mathcal A(X))$ is continuous and the collection $\{ f_U\in A(U) \}_\cU$ is compatible with $\omega$, there exists a locally finite refinement $\mathcal V$ of $\cU$ and a collection $\{ g_V\in A(V) \}_\mathcal V$ satisfying the \emph{overlap condition}
	\begin{equation}\label{overlap}
	\rho_{V\cap W,V}(f_V)= \rho_{V\cap W,W}(f_W),\ \textrm{ whenever } V,W\in \mathcal V\ \textrm{  and } V\cap W\neq \varnothing.
	\end{equation}
\end{lemma}
\begin{proof}
	The sets $ \omega\inv O_{U,f_U}\subset X$ ($U\in\cU$) are open and, since $\{ f_U \}_\cU$ is compatible with $\omega$, cover $X$. Let $\mathcal W$ be a locally finite refinement of $\{ \omega\inv O_{U,f_U} \}_{U\in\cU}$ and $\{(L_W)_U:A(U)\to A(U)\}_U$ ($W\in\mathcal W$) be as in the definition of fine sheaves. We denote $L_W:=(L_W)_X$. For every $W\in \mathcal W$ choose $U\in\cU$ such that $W\subset \omega\inv O_{U,f_U}\subset U$ and let $h_W\in A(X)$ be such that $$\rho_{W,X}(h_W)=\rho_{W,U}(f_U).$$
	
	For each $x\in X$, let $D_x$ be a neighborhood of $x$ satisfying (b) in the same definition. Set
	\[ g_x:=\sum_{W\in\mathcal W}\rho_{D_x,X}\circ L_W(h_W)\in A(D_x). \]

	The collection $\{ g_x\in A(D_x) \}_{x\in X}$ now satisfies (\ref{overlap}). Indeed,
	\begin{align*}
	\rho_{D_x\cap D_y,D_x}(g_x)&=\sum_{W\in \mathcal W}\rho_{D_x\cap D_y,X}(L_W(h_W))\\
	&=\sum_{W\in\mathcal W} \rho_{D_x\cap D_y,D_y}(\rho_{D_y,X}\circ L_W(h_W))\\
	&=\rho_{D_x\cap D_y,D_y}(g_y). 
	\end{align*}
	We pass to a locally finite refinement $\mathcal V$ of $\{ D_x\}_{x\in X}$ and set, for any $V\in\mathcal V$,
	\[
	g_V=\rho_{V,D_x}(g_x)
	\]
	whenever $V\subset D_x$. Clearly
	\begin{align*}
	\rho_{U\cap V,U}(g_U)=& \rho_{U\cap V,D_x\cap D_y}(\rho_{D_x\cap D_y,D_x}(g_x))\\
	=&\rho_{U\cap V,D_x\cap D_y}(\rho_{D_x\cap D_y,D_y}(g_y))=\rho_{U\cap V,V}(g_V)
	\end{align*}
	whenever $U,V\in \mathcal V$ and $U\subset D_x$, $V\subset D_y$.
	
	It remains to show that  $\{g_V\}_\mathcal V$ is compatible with $\omega$. Indeed, for $x\in X$ and $V\in \mathcal V$ a neighborhood of $x$, we have $V\subset D_y\subset \omega\inv O_{U,f_U}\subset U$ for some $D_y$ and $U\in \mathcal U$. Since
	\begin{align*}
 g_V&=\sum_{W\in\mathcal W}\rho_{V,X}(L_W(h_W))=\sum_{W\in\mathcal W}(L_W)_V(\rho_{V,X}(h_W))\\
 &=\sum_{W\in\mathcal W}(L_W)_V(\rho_{V,U}(f_U)))=\rho_{V,U}(f_U),
	\end{align*}
	it follows that $[g_V]_x=[\rho_{V,U}(f_U)]_x=[f_U]_x=\omega(x)$.
\end{proof}
\begin{definition}\label{def:overlap-compatible}
	Let $\omega\in\G(\mathcal A(X))$ be continuous. If $\cU$ is a locally finite open cover, $\{ f_U \}_\cU$ is compatible with $\omega$ and satisfies the overlap condition (\ref{overlap}), we say that $\{f_U \}_\cU$ is \emph{overlap-compatible} with $\omega$.
\end{definition}
\begin{remark}\label{rmk:nicecover}
	Representations $\{ f_U \}_\cU$ of continuous sections $\omega\in\G(\mathcal A(X))$ are stable under passing to refinements. Indeed, if $\mathcal D$ is a refinement of $\mathcal U$ and we set 
	\[
	f_V'=\rho_{V,W}(f_W),
	\]
	for $ V\in \mathcal D$ and $ V\subset W\in \mathcal U,$ the collection $\{f_V'\}_{V\in\mathcal D}$ again represents $\omega$. The same holds true for the overlap condition (\ref{overlap}).
	
	Thus we may always assume that the underlying cover in a representation of $\omega$ is locally finite consists of precompact sets if $X$ is locally compact.
\end{remark}

The vector space of continuous sections over $\mathcal A(X)$ is denoted $\Gamma(\mathcal A(X))$. We remark that there is a canonical linear map
\begin{equation}\label{eq:gamma}
\gamma:A(X)\to \Gamma(\mathcal A(X)),\quad f\mapsto (x\mapsto [f]_x).
\end{equation}

\subsubsection*{Support} Let $\{ A(U);\rho_{U,V} \}_U$ be a presheaf on $X$ and $\mathcal A(X)$ the associated \'etal\'e space. For $\omega\in \G(\mathcal A(X))$, we define 
\[ 
\spt\omega = \overline{\{x\in X: \omega(x)\ne 0 \}}. 
\] 
We say that the section $\omega\in \G(\mathcal A(X))$ has \emph{compact support} if $\spt\omega$ is compact. We denote by $\G_c(\mathcal A(X))$ the vector space of compactly supported (global) sections of $\mathcal A(X)$ and $\Gamma_c(\mathcal A(X))=\Gamma(\mathcal A(X))\cap \G_c(\mathcal A(X))$.

\subsection{Polylipschitz forms and sections}\label{sec:germs}
We move now the discussion from abstract presheaves to presheaves of polylipschitz and homogeneous polylipschitz functions. Let $X$ be a locally compact  metric space and $k\in\N$. Recall the notation introduced in Section \ref{sec:hlip}. We consider two presheaves, namely the collections $\{ \Poly^k(U) \}_U$ and $\{\overline\Poly^k(U)\}_U$ together with the restriction maps
\begin{align*}
&\rho_{U,V}:\Poly^k(V)\to \Poly^k(U),\quad \rho_{U,V}=r_{U,V}^{\otimes_\pi(k+1)},\\
&\bar\rho_{V,U}:\overline\Poly^k(V)\to \overline\Poly^k(U),\quad \bar\rho_{U,V}:=r_{U,V}\otimes_\pi\bar r_{U,V}^{\otimes_\pi k}
\end{align*}
for inclusions $U\subset V$. Note that under the identification described in Section \ref{sec:polylipfns} the map $\rho_{U,V}$ is simply the restriction map $\pi\mapsto \pi|_{U^{k+1}}$. It is not difficult to see (using the corresponding facts for $r_{U,V}$ and $\bar r_{U,V}$) that $\rho_{U,V}$ and $\bar \rho_{U,V}$ satisfy (\ref{eq:restrcommute}) for $U\subset W\subset V$. 
For the purposes of Section \ref{sec:aux}, we note that this property remains true for the quotient maps $\rho_{U,V}$ and $\bar\rho_{U,V}$ for \emph{any} sets $U\subset V\subset X$, in particular also for sets which are not open.

The overlap condition \eqref{overlap} for polylipschitz forms and sections is crucial for defining the action of currents on them. The next proposition establishes this by showing that the presheaves $\{ \Poly^k(U) \}$ and $\{ \overline\Poly^k(U) \}$ are fine.

\begin{proposition}\label{prop:polyfine}
	The presheaves $\{ \Poly^k(U) \}$ and $\{ \overline\Poly^k(U) \}$ are fine, and the maps $\rho_{U,X}$ and $\bar\rho_{U,X}$ are onto.
\end{proposition}
\begin{proof}
	The last claim is immediate since $\rho_{U,X}$ and $\bar\rho_{U,X}$ are quotient maps. Since $X$ is locally compact, any open cover of $X$ admits a locally finite precompact refinement $\cU$. Let $\{ \varphi_U \}_\cU$ be a Lipschitz partition of unity subordinate to $\cU$. For each $U\in\cU$ and $W\subset X$ open, consider the bounded $(k+1)$-linear maps
	\[
	(L_U)_W:\LIP_\infty(W)^{k+1}\to \Poly^k(W),\quad (\pi_0,\ldots,\pi_k)\mapsto \jmath_W(\varphi_U\pi_0,\pi_1,\ldots,\pi_k).
	\]
	The bounded linear maps $\overline{(L_U)}_W:\Poly^k(W)\to \Poly^k(W)$ making the diagram (\ref{diagram:proj}) commute form a presheaf homomorphism of $\{\Poly^k(W)\}$. Since $\spt\varphi_U\subset U$ it follows that $\spt\overline{(L_U)}_W(\pi)\subset U\cap W$ for any $\pi\in\Poly^k(W)$. This shows (a) in the definition of fine presheaves. 
	
	Let $x\in X$ and $D$ be a neighborhood of $x$ meeting only finitely many of the sets in $\cU$. The fact that $\{\varphi_U\}$ is a partition of unity implies that for any $(\pi_0,\ldots,\pi_k)\in\LIP_\infty(W)^{k+1}$
	\begin{align*}
	\sum_{U\in\cU}(L_U)_D(\pi_0,\ldots, \pi_k)=\jmath_D\left(\left(\sum_{U\in\cU}\varphi_U\right)\pi_0,\pi_1,\ldots,\pi_k\right)=\jmath_D(\pi_0,\ldots,\pi_k).
	\end{align*}
	This implies (b) in the definition of fine presheaves.
	
	Since the bounded $(k+1)$-linear maps $Q_W\circ (L_U)_W:\LIP_\infty(W)^{k+1}\to \overline\Poly^k(W)$ satisfy (\ref{eq:locality}) we obtain maps $\overline{(L_U)'}_W:\overline\Poly^k(W)\to\overline\Poly^k(W)$ by Proposition \ref{prop:useful}, for each open $W\subset X$, that form a presheaf homomorphism. Condition (a) now follows from the corresponding statement for $\Poly^k(W)$ and (\ref{eq:relation}). Condition (b) follows as above. 
\end{proof}

We denote by $\Polys^k(X)$ and $\Polys_x^k(X)$ (respectively, $\overline\Polys^k(X),\overline\Polys_x^k(X)$) the \'etal\'e space and stalk at $x$ associated to $\{\Poly^k(U)\}_U$ (respectively for $\{\overline\Poly^k(U)\}_U$). We further denote the various spaces of sections associated to $\Polys^k(X)$ and $\overline\Polys^k(X)$ by
\begin{align*}
&\G^k(X):=\G(\Polys^k(X)),\ \Gamma^k(X):=\Gamma(\Polys^k(X)),\\
&\G^k_c(X):=\G_c(\Polys^k(X)),\ \Gamma_c^k(X):=\Gamma^k(X)\cap \G^k_c(X)\\
&\overline\G^k(X):=\G(\overline\Polys^k(X)),\ \overline\Gamma^k(X):=\Gamma(\overline\Polys^k(X)),\\
&\overline\G^k_c(X):=\G_c(\overline\Polys^k(X)),\ \overline\Gamma_c^k(X):=\overline\Gamma^k(X)\cap \overline\G^k_c(X)
\end{align*}

\begin{definition}
A continuous section in $\overline\Gamma^k(X)$ is a \emph{polylipschitz $k$-form on $X$}. A continuous section of $\Gamma^k(X)$ is called a \emph{$k$-polylipschitz section}.
\end{definition}

We denote by
\begin{align}\label{eq:gammaembed}
\gamma=\gamma_X^k&:\Poly^k(X)\to \Gamma^k(X),\quad \pi\mapsto(x\mapsto [\pi]_x),\\
\bar\gamma=\bar\gamma_X^k&:\overline\Poly^k(X)\to \overline\Gamma^k(X),\quad \bar\pi\mapsto (x\mapsto [\bar\pi]_x),\nonumber
\end{align}
the natural linear maps in \eqref{eq:gamma} associated to the presheaves $\{\Poly^k(U)\}_U$ and $\{\overline\Poly^k(U)\}_U$, respectively.

\subsection{Relationship of polylipschitz forms and polylipschitz sections}
We briefly describe the relationship between $\Gamma^k(X)$ and $\overline\Gamma^k(X)$. The natural operators in this section arise as \emph{linear maps associated to presheaf cohomomorphisms}. We give a general sheaf theoretic construction in Appendix \ref{sec:cohomom}, and establish some of its basic properties there; see Proposition \ref{prop:assocproperties}. Here we apply the results in Appendix \ref{sec:cohomom} without further mention. We assume throughout this section that $X$ and $Y$ are locally compact metric spaces and $k,m\in\N$ are possibly distinct natural numbers.

Using (\ref{eq:restrcomm}) and the uniqueness in diagram \eqref{diagram:proj}, we see that the quotient map  (\ref{eq:quotient}) satisfies
\begin{align}\label{eq:relation}
Q_U\circ\rho_{U,V}=\bar\rho_{U,V}\circ Q_V.
\end{align}
Thus the collection $\{Q_U:\Poly^k(U)\to \overline\Poly^k(U)\}_U$ is a presheaf homomorphism. Let $\mathcal Q$ be the associated linear map
\begin{equation}\label{eq:sectquotient}
\mathcal Q=\mathcal Q_k:\G^k(X)\to \overline\G^k(X).
\end{equation}
The next proposition shows that cohomomorphisms $\displaystyle \{ \Poly^k(U)\}\to \{\overline\Poly^m(V)\}$ satisfying the locality condition \eqref{eq:locality} descend to cohomomorphisms $\displaystyle \{ \overline\Poly^k(U)\}\to \{\overline\Poly^m(V)\}$.

\begin{proposition}\label{prop:hcohomom}
	Let $f:X\to Y$ be a continuous map and $$\displaystyle \varphi=\{\varphi_U:\Poly^k(U)\to \overline\Poly^m(f\inv U)\}$$ an $f$-cohomomorphism, where each $\varphi_U$ is bounded. Assume $\varphi_U\circ\jmath_U$ satisfies \eqref{eq:locality} and let $\overline\varphi_U:\overline\Poly^k(U)\to\overline\Poly^m(f\inv U)$ be the unique bounded linear map satisfying $\varphi_U=\bar\varphi_U\circ Q^k_U$, for each open $U\subset Y$. Then $$ \overline\varphi=\{ \overline\varphi_U:\overline\Poly^k(U)\to\overline\Poly^m(f\inv U) \}$$ is an $f$-cohomomorphism. The linear maps $\varphi^*$ and $\overline\varphi^*$ associated to $\varphi$ and $\overline\varphi$ satisfy
	\[
	\varphi^*=\overline\varphi^*\circ\mathcal Q.
	\]
\end{proposition}
\begin{proof}
	The existence and uniqueness of $\overline\varphi_U$ follows from Proposition \ref{prop:useful}. For open sets $U\subset V\subset Y$, $\varphi_U\circ\rho_{U,V}\circ\jmath_U$ and $\rho_{f\inv U,f\inv V}\circ\varphi_V\circ\jmath_V$ satisfy \eqref{eq:locality}. Since $\varphi$ is an $f$-cohomomorphism, \eqref{eq:relation} and the uniquenenss in diagram \eqref{diagram:proj} implies that
	\begin{align*}
	\overline\varphi_U\circ\bar\rho_{U,V}=\bar\rho_{f\inv U,f\inv V}\circ\overline\varphi_V.
	\end{align*}
	The identity $\varphi^*=\overline\varphi^*\circ\mathcal Q$ follows from the fact that $\varphi_U=\bar\varphi_U\circ Q_U$ for each open $U\subset Y$.
\end{proof}

\subsection{Sequential convergence on $\Gamma_c^k(X)$ and $\overline\Gamma_c^k(X)$} To study sequential continuity of linear maps between polylipschitz forms and sections, we introduce a notion of sequential convergence on $\Gamma_c^k(X)$ and $\overline\Gamma_c^k(X)$. Recall that, by Proposition \ref{prop:polyfine} and Lemma \ref{lem:overlap}, polylipschitz forms and sections admit overlap-compatible representations indexed by a locally finite precompact open cover; see also Remark \ref{rmk:nicecover}.

\begin{definition}\label{def:polyformconv}
	We say a sequence $(\omega_n)$ in $\Gamma_c^k(X)$ \emph{convergences to $\omega\in\Gamma_c^k(X)$}, denoted $\omega_n\to\omega$ in $\Gamma_c^k(X)$, if there exists a compact set $K\subset X$, a locally finite precompact open cover $\cU$ of $X$, and, for each $n\in \N$, a collection $\{\pi^n_U \}_{U\in\cU}$ overlap-compatible with $\omega_n-\omega$ having the following properties:
	\begin{itemize}
		\item[(1)] $\spt(\omega_n-\omega)\subset K$ for each $n\in\N$, and
		\item[(2)] $\pi^n_U\to 0$ in $\Poly^k(U)$ for each $U\in \cU$.
	\end{itemize}
	Convergence of a sequence $(\bar\omega_n)$ in $\overline\Gamma_c^k(X)$ to $\bar\omega\in\overline\Gamma_c^k(X)$ is defined analogously.
\end{definition}

For metric spaces $X$ and $Y$, let $A=\{ A(U) \}_U$ and $B=\{ B(V) \}_V$ denote either of the presheaves $\{ \Poly^k(U) \}_U$ or $\{ \overline\Poly^m(U) \}_U$ on $X$ and $Y$, respectively, and let
\[
L:\Gamma_c(\mathcal B(Y))\to \Gamma_c(\mathcal A(X))
\]
be a linear map. We say that $L$ is \emph{sequentially continuous} if
\[
L(\omega_n)\to L(\omega)\textrm{ in }\Gamma_c(\mathcal A(X)) \textrm{ whenever }\omega_n\to \omega \textrm{ in }\Gamma_c(\mathcal B(X)).
\]

The natural quotient map from polylipschitz sections to polylipschitz forms is sequentially continuous.
\begin{proposition}
The map $\mathcal Q:\Gamma_c^k(X)\to \overline\Gamma_c^k(X)$ is sequentially continuous.
\end{proposition}

This follows immediately from an abstract result on sequential continuity of linear maps associated to cohomomorphisms.

\begin{proposition}\label{prop:seqcont}
	Suppose $f:X\to Y$ is a proper continuous map, and
	\[
	\varphi=\{ \varphi_U:B(U)\to A(f\inv(U)) \}
	\]
	an $f$-cohomomorphism, where $A=\{ A(U) \}_U$ and $B=\{ B(V) \}_V$ denote either of the presheaves $\{ \Poly^k(U) \}_U$ or $\{ \overline\Poly^m(U) \}_U$ on $X$ and $Y$, respectively. 
	
	If $\varphi_U$ is bounded and sequentially continuous for each open $U\subset Y$, then the associated linear map
	\[
	\varphi^*:\Gamma_c(\mathcal B(Y))\to \Gamma_c(\mathcal A(X))
	\]
	is sequentially continuous.
	
	If $A=\{ \overline\Poly^m(U) \}_U$, $B=\{ \Poly^k(U) \}_U$, and $\varphi_U\circ\jmath_U$ satisfies \eqref{eq:locality} for each open $U\subset Y$, then the linear map
	\[
	\overline\varphi^*:\overline\Gamma_c^k(Y)\to\overline\Gamma_c^m(X)
	\]
	in Proposition \ref{prop:hcohomom} is sequentially continuous.
\end{proposition}
\begin{proof}
	We prove the first claim in case $A=\{ \overline\Poly^m(U)\}_U$ and $B=\{ \Poly^k(U) \}$. The other cases are analogous.
	
	Since $\varphi^*$ is linear it suffices to prove sequential continuity at the origin. Let $\omega_n\to 0$ in $\Gamma_c^k(Y)$, and let $K\subset X$, $\mathcal U$ and $\{ \pi_U^n \}_\mathcal{U}$ be as in Definition \ref{def:polyformconv}. For each $V\in f\inv \cU$ choose $U_V\in\cU$ such that $V=f\inv\cU$. Then the collection $$\{\varphi_{U_V}(\pi_{U_V}^n) \}_{f\inv\mathcal U}$$ is overlap-compatible with $\varphi^*\omega_n$ for each $n\in\N$; cf.~proof of Proposition \ref{prop:assocproperties}(2). Since $\spt\varphi^*\omega^n\subset f\inv K$ and $\varphi_{U_V}$ is sequentially continuous for each $n\in\N$ and $V\in f\inv \cU$, we have that $\varphi^*\omega_n\to 0$ in $\Gamma^m_c(X)$.
	
To prove the last claim assume that $A=\{ \overline\Poly^m(U) \}_U$, $B=\{ \Poly^k(U) \}_U$, and that $\varphi_U\circ\jmath_U$ satisfies \eqref{eq:locality} for each $U\subset Y$. By Proposition \ref{prop:useful} the unique map $\overline\varphi_U$ satisfying $\varphi_U=\overline\varphi_U\circ Q_U$ is sequentially continuous. Thus, the associated linear map $	\overline\varphi^*:\overline\Gamma_c^k(Y)\to\overline\Gamma_c^m(X) $ is sequentially continuous.
\end{proof}

\begin{remark}\label{rmk:bilincont}
	Propositions \ref{prop:seqcont} and \ref{prop:hcohomom} have natural bilinear analogues in the situation of Remark \ref{rmk:bilinear}. The proofs are similar and we omit the details.
\end{remark}

\section{Exterior derivative, pull-back, and cup-product of polylipschitz forms}\label{sec:aux}
In this section we introduce the \emph{exterior derivative, pull-back} and \emph{cup product} on polylipschitz forms and sections. We prove that they are sequentially continuous with respect to a natural notion of sequential convergence on $\Gamma_c^k(X)$ and $\overline\Gamma_c^k(X)$; cf.~Definition \ref{def:polyformconv}. The results in this section are important for applications to currents, and will be used extensively in \cite{teripekka2}.

\subsection{Pull-back} Let $f:X\to Y$ be a Lipschitz map. Consider the $f$-cohomomorphism $\varphi:=\{ f^\#_U:\Poly^k(U)\to \Poly^k(f\inv U) \}$, where $ f^\#_U$ is given by
\[
\pi\mapsto \pi\circ(f|_{f\inv U}\times\cdots\times f|_{f\inv U}).
\]
The maps $\varphi':=\{Q_{f\inv U}\circ f^\#_U:\Poly^{k}(U)\to \overline\Poly^{k}(f\inv U)\}$ form an $f$-cohomomorphism and $Q_{f\inv U}\circ f^\#_U\circ\jmath_U$ satisfies \eqref{eq:locality} for each $U\subset Y$. By Proposition \ref{prop:hcohomom} the linear maps $f^\#:\G^k(Y)\to \G^k(X)$ and  $\bar f^\#:\overline\G^k(Y)\to \overline\G^k(X)$ associated to $\varphi$ and $\overline\varphi'$, respectively, satisfy
\[
\mathcal Q\circ  f^\#=\bar f^\#\circ\mathcal Q.
\]
We refer to the linear maps $f^\#$ and $\bar f^\#$ as \emph{pull-backs}. If $f$ is proper, Proposition \ref{prop:seqcont} implies that $f^\#$ and $\bar f^\#$ are sequentially continuous.

If $E\subset X$, the pull-backs $\iota_E^\#$ and $\bar \iota_E^\#$ given by the construction above for the inclusion map $\iota_E:E\hookrightarrow X$ are called \emph{restrictions to $E$}. We denote 
\[
\omega|_E:=\iota_E^\#(\omega)\textrm{ and }\bar\omega|_E:=\bar\iota_E^\#(\bar\omega)
\]
for $\omega\in \G^k(X)$ and $\bar\omega\in\overline\G^k(X)$. Note that, when $E\subset X$ is closed, the inclusion $\iota_E$ is a proper map. Thus the restriction operator to closed sets is sequentially continuous.

\subsection{Exterior derivative}
As in Alexander-Spanier cohomology (see \cite[Section 5.26]{war83}) we define a linear map 
\[ 
d=d_X^k:\LIP_\infty(X^{k+1})\to \LIP_\infty(X^{k+1})
\] 
by
\begin{equation}\label{extd} 
d_X^k\pi(x_0,\ldots,x_{k+1})=\sum_{j=0}^{k+1}(-1)^j\pi(x_0,\ldots,\hat x_j,\ldots x_{k+1})
\end{equation}
for $\pi\in \LIP_\infty(X^{k+1})$ and $x_0,\ldots,x_k\in X$. It is a standard exercise to show that

\begin{equation}\label{eq:standardex}
d_X^{k+1}\circ d_X^k=0
\end{equation}

\begin{lemma}\label{extdcont}
	For each $\pi\in\Poly^k(X)$ and each open set $V\subset X$, we have
	\begin{equation}\label{x4}
	L_{k+1}(d\pi; V)\le (k+2)L_k(\pi;V).
	\end{equation}
Thus $d$ defines a bounded linear map $d:\Poly^k(X)\to \Poly^{k+1}(X)$ which, moreover, is sequentially continuous.
\end{lemma}

It follows from Lemma \ref{extdcont} that $\{d_U^k:\Poly^k(U)\to \Poly^{k+1}(U)\}$, and consequently $\{Q_U^{k+1}\circ d_U^k:\Poly^k(U)\to \overline\Poly^{k+1}(U) \}$, are presheaf homomorphisms, and $Q_U^{k+1}\circ d_U^k\circ\jmath_U$ satisfies \eqref{eq:locality} for each open $U\subset X$. By Propositions \ref{prop:hcohomom} and \ref{prop:seqcont} we obtain sequentially continuous associated linear maps
\[
\bar d :=\bar d_X^k:\overline\G^k(X)\to \overline\G^{k+1}(X)
\]
and
\[
d:=d_X^k:\G^k(X)\to \G^{k+1}(X),
\]
called the \emph{exterior derivative} of polylipschitz forms and sections, respectively. 

\begin{remark}
In fact the identity
\begin{equation}\label{restprop}
\bar d_{A}^k(\bar \rho_{A,B}(\bar\pi))= \bar \rho_{A,B}(\bar d_B^k\bar \pi)\quad \textrm{ for all } \bar\pi\in \overline \Poly^k(B).
\end{equation}
holds for \emph{any} sets $A,B\subset X$. Thus, if $E\subset X$ and $\bar\omega\in \overline\G^k(X)$ we have $\bar d(\bar d\bar\omega)=0$ and $(\bar d\bar \omega)|_E=\bar d(\bar\omega|_E).$ The first identity follows from (\ref{eq:standardex}) and Proposition \ref{prop:assocproperties}(4), while the second is implied by (\ref{restprop}).

The same identities hold for restrictions and the exterior derivative of polylipschitz sections.
\end{remark}
These properties of the exterior derivative and restriction are used in the sequel without further mention.

We conclude this subsection with the proof of Lemma \ref{extdcont}. For the proof, the following expression, for $(\pi_0,\ldots,\pi_k)\in\LIP_\infty(X)^{k+1}$ and $\pi=\pi_0\otimes\cdots\otimes\pi_k$, will be useful.
\begin{equation}\label{eq:simplebdry}
d_X^k(\pi_0\otimes\cdots\otimes\pi_k)=\sum_{l=0}^{k+1}(-1)^l\pi_0\otimes\cdots\otimes \pi_{l-1}\otimes 1\otimes\pi_{l}\otimes\cdots\otimes\pi_k
\end{equation}
\begin{proof}[Proof of Lemma \ref{extdcont}]
	Let $ \pi_0,\ldots,\pi_k\in \LIP_\infty(X)$ and $\pi=\pi_0\otimes\cdots\otimes\pi_k\in \Poly^k(X)$. By (\ref{eq:simplebdry}) we have
	\[ 
	L_k(d_X^k\pi;V)\le \sum_{l=0}^{k+1}L(\pi_0|_V)\cdots L(\pi_k|_V)=(k+2)L_k(\pi;V). 
	\] 
	Thus, by the subadditivity of $L_k(\cdot; V)$, we have, for each $\pi\in \Poly^k(X)$ and  $\displaystyle (\pi_0^j, \ldots,\pi_k^j)_j\in \Rep(\pi)$, the estimate
	\begin{align*}
	L_{k+1}(d_X^k\pi;V)\le \sum_j(k+2)L(\pi_0^{j}|_V)\cdots  L(\pi_k^j|_V).
	\end{align*}
	Taking infimum over all such representatives yields (\ref{x4}).
	
	To show that $d_X^k$ is sequentially continuous suppose $\pi^n\to 0$ in $\Poly^k(X)$ and let $V=V_0\times\cdots\times V_k$ be compact.  For each $n\in\N$, let
	\[
	\pi^n=\sum_j^\infty\pi_0^{j,n}\otimes\cdots\otimes\pi_k^{j,n}
	\]
	be a representation of $\pi^n$ satisfying (1) and (2) in Definition \ref{conv}. Then 
	\begin{equation}\label{eq:repr}
	d_X\pi^n=\sum_{l=0}^{k+1}\sum_j^\infty(-1)^l\pi_0^{j,n}\otimes\cdots\otimes\pi_{l-1}^{j,n}\otimes 1 \otimes\pi_l^{j,n}\cdots\otimes\pi_k^{j,n} 
	\end{equation}
	is a representation of $d_X^k\pi^n$.
	
	To show that condition (1) in Definition \ref{conv} is satisfied it suffices to observe that, for every $j\in \N$ and $l\in \{ 0,\ldots,k+1\}$, we have 
	\[
	\sup_n L((-1)^l\pi_0^{j,n})L(\pi_{1}^{j,n})\cdots L(\pi_k^{j,n})=\sup_n L(\pi_0^{j,n})\cdots L(\pi_k^{j,n}).
	\]
	Moreover, for each compact set  $V_0\times\cdots\times V_k\subset X^{k+1}$, we have
	\begin{align*}
	&\sum_{l=0}^{k+1}\sum_j^\infty\|(-1)^{l}\pi_0^{j,n}|_{V_0}\|_\infty\cdots \|\pi_k^{j,n}|_{V_k}\|_\infty 
	\le  (k+2)\sum_j^\infty\|\pi_0^{j,n}|_{V_0}\|_\infty\cdots \|\pi_k^{j,n}|_{V_k}\|_\infty.
	\end{align*}
	Thus condition (2) in Definition \ref{conv} is satisfied by the representation (\ref{eq:repr}). It follows that $d_X\pi^n\to 0$ in $\Poly^{k+1}(X)$.
\end{proof}

\subsection{Cup product}

Given polylipschitz functions $\pi\in \Poly^k(X)$ and $\sigma\in \Poly^m(X)$, their cup product is the function $\pi\smile\sigma:X^{k+m+1}\to\R$,
\[ 
(x_0,x_1,\ldots x_{k+m})\mapsto \pi(x_0,x_1,\ldots,x_k)\sigma(x_0,x_{k+1},\ldots,x_{m+k}). 
\]

If $(\pi_0^j,\ldots,\pi_k^j)\in \Rep(\pi)$ and $(\sigma_0^j,\ldots,\sigma_k^j)\in \Rep(\sigma)$ are representations of polylipschitz functions $\pi\in \Poly^k(X)$ and $\sigma\in \Poly^m(X)$, respectively, we observe that $(\pi_0^j\sigma_0^i,\pi_1^j,\ldots,\pi_k^j,\sigma_1^i,\ldots,\sigma_m^i)_{i,j}$ is a representation of $\pi\smile\sigma$ and that $\pi\smile \sigma\in \Poly^{k+m}(X)$. The proof of the next Lemma follows from Definition \ref{conv} and straightforward calculations and estimates. We omit the details.

\begin{lemma}\label{lem:cupprodcont}
The cup-product $\cdot\smile\cdot:\Poly^k(X)\times\Poly^m(X)\to \Poly^{k+m}(X)$ is a sequentially continuous bounded bilinear map.
\end{lemma}

The collection $\{ \smile:\Poly^k(U)\times\Poly^m(U)\to \Poly^{k+m}(U)\}$ is a bilinear presheaf homomorphism, and we note that
$$Q_U^{k+m}\circ\smile\circ(\jmath_U^k\times\jmath_U^m):\LIP_\infty(U)^{k+1}\times\LIP_\infty(U)^{m+1}\to \overline\Poly^{k+m}(U)$$
satisfies the bi-linear analogue of \eqref{eq:locality} for each open $U\subset X$. By Lemma \ref{lem:cupprodcont} and Remark \ref{rmk:bilincont} (see also Remark \ref{rmk:bilinear}) we obtain bilinear maps
\[
\smile:\overline\G^k(X)\times\overline\G^m(X)\to\overline\G^{k+m}(X)
\]
and
\[
\smile:\G^k(X)\times\G^m(X)\to\G^{k+m}(X),
\]
called the \emph{cup product} of polylipschitz forms and sections, respectively. Note that
\[
\spt(\bar\omega\smile\bar\sigma)\subset \spt \bar\omega\cap \spt\bar\sigma,\quad \overline\Gamma^k(X)\smile\overline\Gamma^m(X)\subset \overline\Gamma^{k+m}(X),
\]
see Remark \ref{rmk:bilinear}. We record the following standard identities for cup products, pull-backs and the exterior derivative; cf. \cite{mas78,mas67}.

\begin{lemma}\label{basic}
	Let $X$ and $Y$ be metric spaces, $f:X\to Y$ a Lipschitz map. Let $\bar\alpha\in\overline\G^k(Y)$ and $\bar\beta\in\overline\G^m(Y)$. Then 
	\begin{itemize}
		\item[(a)] $\bar f^\#(\bar\alpha\smile\bar\beta)=(\bar f^\#\bar \alpha)\smile (\bar f^\#\bar\beta),$ and
		\item[(b)] $\bar  d(\bar\alpha\smile\bar\beta)=\bar d\bar\alpha\smile\bar\beta+(-1)^k\bar\alpha\smile \bar d\bar\beta.$
	\end{itemize}
The same identities hold for $\alpha\in\G^k(Y)$ and $\beta\in\G^m(Y)$.
\end{lemma}

\section{Metric currents as the dual of polylipschitz forms}\label{sec:currentdual}
In this section $X$ is a locally compact metric space and $k\in\N$. We prove that metric currents act sequentially continuously on the space of polylipschitz forms. Recall the natural maps (\ref{eq:jmath}), (\ref{eq:barjmath}) and (\ref{eq:gammaembed}) and denote
\begin{equation}
\label{eq:iota}
\iota=\iota_X^k:\sD^k(X)\to \Gamma_c^k(X),\quad \iota:=\gamma\circ\jmath|_{\sD^k(X)},
\end{equation}
and
\begin{equation}
\bar\iota=\bar\iota_X^k:\sD^k(X)\to \overline\Gamma_c^k(X),\quad \bar\iota:=\bar\gamma\circ\bar\jmath|_{\sD^k(X)},
\end{equation}
It follows from the respective definitions of sequential convergence that $\iota$ and $\bar\iota$ are sequentially continuous.
\begin{theorem}\label{ext1}
	For each $T\in\sD_k(X)$, there exists a unique sequentially continuous linear map $\widehat T:\overline\Gamma_c^k(X)\to \R$ satisfying $T=\widehat T\circ\bar\iota$.
\end{theorem}

We use an auxiliary result for the proof of Theorem \ref{ext1}. For the next lemma, let $T\in \sD_k(X)$ be a metric current, $\varphi\in\LIP_c(X)$, and $U$ an open set containing $K:=\spt\varphi$. We define the $(k+1)$-linear map
\begin{equation}\label{eq:map}
T_\varphi^U:\LIP_\infty(U)^{k+1}\to \R,\quad (\pi_0,\ldots,\pi_k)\mapsto T(\varphi\pi_0,\widetilde\pi_1,\ldots,\widetilde\pi_k)
\end{equation}
for any extension $\widetilde\pi_l\in \LIP_\infty(X)$ of $\pi_l$. By (\ref{eq:bound}) and the discussion preceding it we have that the map $T_U$ is well-defined and bounded, with the bound
\begin{equation}\label{bound}
|T_\varphi^U(\pi_0,\ldots,\pi_k)|\le 2CL(\varphi)L(\pi_0|_K)\cdots L(\pi_k|_K)
\end{equation}
for $(\pi_0,\ldots,\pi_k)\in\LIP_\infty(U)^{k+1}$.

\begin{lemma}\label{lem:need}
The bounded $(k+1)$-linear map $T_\varphi^U:\LIP_\infty(U)^{k+1}\to \R$ in (\ref{eq:map})
descends to a unique sequentially continuous bounded linear map $\overline{T_\varphi^U}:\overline\Poly^k(U)\to \R$ satisfying $\overline{T_\varphi^{U}}\circ\bar\jmath=T_\varphi^U$.
\end{lemma}
\begin{proof}
By the locality properties of currents, we have that
\[
T_\varphi^U(\pi_0,\ldots,1,\ldots,\pi_k)=0,\quad (\pi_0,\ldots,1,\ldots,\pi_k)\in \LIP_\infty(U)^{k+1}
\]
if one of the functions $\pi_l$ is the constant one for $l=1,\ldots, k$; cf.~Definition \ref{def:locality}. By Proposition \ref{prop:useful} $T_\varphi^U$ descends to a bounded linear map $\overline{T_\varphi^U}:\overline\Poly^k(U)\to \R$, satisfying $T_\varphi^U=\overline{T_\varphi^U}\circ\bar\jmath$, and the sequential continuity of $\overline{T_\varphi^U}$ is implied by the sequential continuity of the bounded linear map $\overline A:\Poly^k(U)\to \R$ for which $T_\varphi^U=\overline A\circ\jmath$. To prove sequential continuity of $\overline A$, suppose that the sequence $(\pi^n)_n$ converges to zero in $\Poly^k(U)$. It suffices to prove that each subsequence of $\{\overline A(\pi^n)\}_n$ has a further subsequence converging to zero.

Since $\pi^n\to 0$ in $\Poly^k(U)$ there are representatives $\displaystyle (\pi_0^{j,n},\ldots,\pi_k^{j,n})_j\in \Rep(\pi^n)$ satisfying 
	\[ 
	\sum_j^\infty\sup_n L(\pi_0^{j,n})\cdots L(\pi_k^{j,n})<\infty
	\]
	and
	\[
	\lim_{n\to\infty}\sum_j^\infty\|\pi_0^{j,n}|_K\|_\infty\cdots \|\pi_k^{j,n}|_K\|_\infty=0. 
	\] 
	For each $j\in\N$, we denote $\displaystyle t_j:=\sup_nL(\pi_0^{j,n})\cdots L(\pi_k^{j,n})$.
	We may assume $L(\pi_l^{j,n})\ne 0$ for all $l=0,\ldots,k$ and $j,n\in\N$. For each $j,n\in\N$ and $l=0,\ldots, k$, let 
	\[
	\sigma_l^{j,n}=a_l^{j,n}\pi_l^{j,n},
	\]
	where 
	\[
	a_l^{j,n}= \frac{\big(L(\pi_0^{j,n})\cdots L(\pi_k^{j,n})\big)^{1/(k+1)}}{L(\pi_l^{j,n})}.
	\]
	Then
	\begin{equation*}\label{key}
	\sigma_0^{j,n}\otimes\cdots\otimes\sigma_k^{j,n}=\pi_0^{j,n}\otimes\cdots\otimes\pi_k^{j,n}
	\end{equation*}
	which implies
	\begin{align}\label{1}
	T(\varphi\pi_0^{j,n},\pi_1^{j,n},\ldots,\pi_k^{j,n})&=\overline A(\pi_0^{j,n}\otimes\cdots\otimes\pi_k^{j,n})=\overline A(\sigma_0^{j,n}\otimes\cdots\otimes\sigma_k^{j,n})\nonumber\\
	&=T(\varphi\sigma_0^{j,n},\sigma_1^{j,n},\ldots,\sigma_k^{j,n}).
	\end{align}
	For each $j,n\in\N$ and $l=0,\ldots,k$, we have
	\begin{align*}
	L(\sigma_l^{j,n})&=[L(\pi_0^{j,n})\cdots L(\pi_k^{j,n})]^{1/(k+1)}\le t_j^{1/(k+1)}.
	\end{align*}
	By the Arzela-Ascoli theorem and a diagonal argument, there exists a subsequence and $\sigma_l^j\in\LIP_\infty(U)$ for which $\sigma_l^{j,n}\to \sigma_l^j$ in $\LIP_\infty(U)$ as $n\to\infty$, for each $j\in\N$ and $l=0,\ldots,k$. Thus
	\begin{equation}\label{2}
	\lim_{n\to\infty}T(\varphi\sigma_0^{j,n},\sigma_1^{j,n},\ldots,\sigma_k^{j,n})=T(\varphi\sigma_0^j,\sigma_1^j,\ldots,\sigma_k^j).
	\end{equation}
	Moreover, for fixed $j\in\N$, we have that
	\begin{align*}
	\lim_{n\to\infty}\|\sigma_0^{j,n}|_K\|_\infty\cdots \|\sigma_k^{j,n}|_K\|_\infty=\lim_{n\to\infty}\|\pi_0^{j,n}|_K\|_\infty\cdots \|\pi_k^{j,n}|_K\|_\infty=0
\end{align*}
and 
\begin{align*}
\|\sigma^{j,n}_l|_K\|_\infty=a_l^{j,n}\|\pi_l^{j,n}|_K\|_\infty\le t_j^{1/(k+1)}\textrm{ for all $n$ and }l=0,\ldots,k.
\end{align*}
	It follows that, up to passing to a further subsequence, there is, for each $j\in\N$, an index $l=0,\ldots,k$ for which $\displaystyle \lim_{n\to\infty}\|\sigma_l^{j,n}|_K\|_\infty=0$. Consequently, for each $j\in\N$, there exists $l=0,\ldots,k$ for  which
	\begin{equation}\label{3}
	\sigma_l^j|_K=0.
	\end{equation}
	The locality of $T$ together with (\ref{1}),(\ref{2}), and (\ref{3}) now implies that
	\begin{equation}\label{truekey}
	\lim_{n\to\infty}T(\varphi\pi_0^{j,n},\pi_1^{j,n},\ldots,\pi_k^{j,n})=0
	\end{equation}
	for each $j\in\N.$ The estimate (\ref{bound}) yields
	\[
	|T(\varphi\pi_0^{j,n},\pi_1^{j,n},\ldots,\pi_k^{j,n})|\le Ct_j 
        \]
        for each $j\in\N$ and for some constant $C>0$. The dominated convergence theorem now implies that
	\begin{align*}
	\lim_{n\to\infty}\overline A(\pi^n)=&\lim_{n\to\infty}\sum_j^\infty T(\pi_0^{j,n},\ldots,\pi_k^{j,n})
	=\sum_j^\infty\lim_{n\to\infty}  T(\pi_0^{j,n},\ldots,\pi_k^{j,n})=0.
	\end{align*}
This concludes the proof.
\end{proof}

\begin{remark}\label{rmk:functorial}
By the multilinearity of currents and the uniqueness in Lemma \ref{lem:need}, we observe the following functorial property: If $\varphi,\psi\in\LIP_c(X)$, $U\subset V$  are open sets, and $\spt\varphi\cup\spt\psi\subset U$, we have that
	\begin{align*}
	\overline{T_{\varphi+\psi}^U}=\overline{T_\varphi^U}+\overline{T_\psi^U}\textrm{ and } \overline{T_\varphi^V}=\overline{T_\varphi^U}\circ\bar \rho_{U,V}
	\end{align*}
	for each $T\in\sD_k(X)$.
\end{remark}

\begin{proof}[Proof of Theorem \ref{ext1}]
Let $T\in\sD_k(X)$. We define $\widehat T:\overline\Gamma_c^k(X)\to V$ as follows. Let $\omega\in\overline\Gamma_c^k(X)$ be overlap-compatible with $\{ \bar\pi_U\in\overline\Poly^k(U) \}_{U\in\mathcal U}$ for a locally finite precompact open cover; cf. Remark \ref{rmk:nicecover}. Let $\{ \varphi_U \}_\mathcal U$ be a Lipschitz partition of unity subordinate to $\mathcal U$ and set
\begin{equation}\label{eq:define}
\widehat T(\omega):=\sum_{U\in\mathcal U}\overline{T_{\varphi_U}^U}(\bar\pi_U),
\end{equation}
where $T_{\varphi_U}^U$ is as in Lemma \ref{lem:need}. Note that, since $\spt\omega$ is compact, only finitely many $\pi_U$ are nonzero and the sum has only finitely many nonzero terms.

We prove that $\widehat T$ is well-defined. Let $\{ \bar\sigma_V \}_{V\in\mathcal V}$ be overlap-compatible with $\omega$, and $\{\psi_V\}_{\mathcal V}$ a Lipschitz partition of unity subordinate to $\mathcal V$. Let $\mathcal W$ be a locally finite refinement of $\mathcal U\cap \mathcal V=\{ U\cap V: \ U\in\mathcal U,\ V\in \mathcal V\}$ with the property that, whenever $U\in\mathcal U,\ V\in\mathcal V,\ W\in\mathcal W$, and $W\subset U\cap V$, we have that
\[
\bar\rho_{W,U}(\bar\pi_U)=\bar\rho_{W,V}(\bar\sigma_V);
\]
cf. Remark \ref{rmk:nicecover} and the discussion after it. In particular
\begin{equation}\label{eq:ok}
\bar\rho_{U\cap V\cap W,U}(\bar\pi_U)=\bar\rho_{U\cap V\cap W,V}(\bar\sigma_V)
\end{equation}
for all $U\in\mathcal U,\ V\in\mathcal V$ and $W\in \mathcal W$.

Let $\{\theta_W\}_{\mathcal W}$ be a Lipschitz partition of unity subordinate to $\mathcal W$. By the functorial properties in Remark \ref{rmk:functorial}, and (\ref{eq:ok}), we have that
\begin{align*}
\sum_{U\in\mathcal U}\overline{T_{\varphi_U}^U}(\pi_U)&=\sum_{U\in\mathcal U}\sum_{V\in\mathcal V}\sum_{W\in\mathcal W}\overline{T_{\varphi_U\psi_V\theta_W}^U}(\bar\pi_U)\\
&=\sum_{U\in\mathcal U}\sum_{V\in\mathcal V}\sum_{W\in\mathcal W}\overline{T_{\varphi_U\psi_V\theta_W}^{U\cap V\cap W}}(\bar\rho_{U\cap V\cap W,U}(\bar\pi_U))\\
&=\sum_{U\in\mathcal U}\sum_{V\in\mathcal V}\sum_{W\in\mathcal W}\overline{T_{\varphi_U\psi_V\theta_W}^{U\cap V\cap W}}(\bar\rho_{U\cap V\cap W,V}(\bar\sigma_V))\\
&=\sum_{V\in\mathcal V}\overline{T_{\varphi_V}^V}(\bar\sigma_V).
\end{align*}
Note that all the sums above have only finitely many non-zero summands. This shows that $\widehat T$ is well-defined. 

To prove that $\widehat T:\overline\Gamma_c^k(X)\to\R$ is sequentially continuous, let $\omega_n\to 0$ in $\overline\Gamma_c^k(X)$. By Definition \ref{def:polyformconv} there is a compact set $K\subset X$, a locally finite precompact open cover $\mathcal U$ and $\{\bar\pi_U^n\}_\mathcal{U}$ overlap-compatible with $\omega_n$, for each $n\in\N$, so that $\spt\omega_n\subset K$ for each $n\in\N$ and
\[
\bar\pi_U^n\to 0\textrm{ in }\overline\Poly^k(U)
\]
as $n\to\infty$, for each $U\in\mathcal U$. 

Since $K$ is compact and $\mathcal U$ is locally finite, the collection $$\mathcal U_K:=\{ U\in\mathcal U: U\cap K\ne\varnothing\}$$ is finite. Moreover, for each $n\in\N$ and $U\notin\mathcal U_K$, we have that $\bar\pi_U^n=0$, since otherwise $\omega_n(x)=[\bar\pi_U^n]_x\ne 0$ for some $x\notin K$. It follows by Lemma \ref{lem:need} that 
\[
\lim_{n\to \infty}\widehat T(\omega_n)=\sum_{U\in\mathcal U_K}\lim_{n\to\infty}\overline{T_{\varphi_U}^U}(\bar\pi_U^n)=0.
\]

To prove the factorization, suppose $(\pi_0,\ldots,\pi_k)\in\sD^k(X)$ and $\{ \varphi_U \}_\mathcal{U}$ is a Lipschitz partition of unity subordinate to a locally finite precompact open cover $\mathcal U$. Then $\{\bar\pi_U:=\pi_0|_U\otimes\overline{\pi_1|_U}\otimes\cdots\otimes\overline{\pi_k|_U}\}_\mathcal U$ is overlap-compatible with $\bar\iota(\pi_0,\ldots,\pi_k)$. Proposition \ref{lem:need} and (\ref{eq:map}) now imply that
\[
\widehat T(\bar\imath(\pi_0,\ldots,\pi_k))=\sum_{U\in\mathcal U}\overline{T_{\varphi_U}^U}(\bar\pi_U)=\sum_{U\in\mathcal U}T(\varphi_U\pi_0,\pi_1,\ldots,\pi_k)=T(\pi_0,\ldots,\pi_k).
\]

For uniqueness, let $A:\overline\Gamma_c^k(X)\to\R$ be linear and sequentially continuous, and $A\circ\bar\iota=T$. Let $\omega\in\overline\Gamma_c^k(X)$ be overlap-compatible with $\{ \bar\pi_U \}_\mathcal U$, and let $\{\varphi_U\}_\mathcal U$ be a Lipschitz partition of unity subordinate to $\mathcal U$. We may assume that $\bar\pi_U\in\overline\Poly^k(X)$ for each $U\in\mathcal U$ by Lemma \ref{ext} and the surjectivity of \ref{eq:quotient}.

Note that $\varphi_U\smile\omega=\bar\gamma(\varphi_U\smile\bar\pi_U)$ and, by linearity,
\[
A(\omega)=\sum_{U\in\cU}A(\varphi_U\smile\omega)=\sum_{U\in\cU}A(\bar\gamma(\varphi_U\smile\bar\pi_U)).
\]
For each $U\in \mathcal U$, the linear map
\[
A_U:\overline\Poly^k(U)\to \R,\quad A_U=A(\bar\gamma(\varphi_U\smile\cdot))
\]
is bounded and satisfies
\begin{align*}
A_U\circ\bar\jmath(\pi_0,\ldots,\pi_k)&=A(\bar\gamma((\varphi_U\pi_0)\otimes\bar\pi_1\otimes\cdots\otimes\bar\pi_k ))\\
&=A(\bar\gamma\circ\bar\jmath(\varphi_U\pi_0,\pi_1,\ldots,\pi_k))=T(\varphi_U\pi_0,\pi_1,\ldots,\pi_k)
\end{align*}
for $(\pi_0,\ldots,\pi_k)\in \sD^k(X)$; cf. (\ref{eq:iota}). The uniqueness in Lemma \ref{lem:need} implies that $A_U=\overline{T_{\varphi_U}^U}$. Hence
\[
A(\omega)=\sum_{U\in\cU}A(\bar\gamma(\varphi_U\smile\bar\pi_U))=\sum_{U\in\mathcal U}\overline{T_{\varphi_U}^U}(\bar\pi_U)=\widehat T(\omega).
\]
The proof is complete.
\end{proof}

The next proposition establishes \eqref{main1:partial}.

\begin{proposition}\label{prop:bdry}
	Let $T\in\sD_{k}(X)$ be a $k$-current on $X$. Let $\widehat T:\overline\Gamma_c^{k}(X)\to\R$ and $\widehat{\partial T}:\overline\Gamma_{c}^{k-1}(X)\to\R$ be extensions of $T$ and $\partial T$, respectively. Then we have \[ \widehat{\partial T}(\omega)=\widehat T(\bar d\omega)\] for each $\omega\in \overline\Gamma_{c}^{k-1}(X)$.
\end{proposition}\label{bdry}

\begin{proof}
	Since $\bar d$ is sequentially continuous it follows that $\widehat T\circ \bar d$ defines an element in $\Gamma_c^{k-1}(X)^*$. Since the extension in Theorem \ref{ext1} is unique, it suffices to show that \[ \widehat T(\bar d(\bar\iota(\pi_0,\ldots,\pi_{k-1}))=\partial T(\pi_0,\ldots,\pi_{k-1}) \] for $(\pi_0,\ldots,\pi_{k-1})\in \sD^{k-1}(X)$.

	The expression (\ref{eq:simplebdry}) shows that, for each open $U\subset X$, the $l^{th}$ term in the sum (\ref{eq:simplebdry})  is constant the $x_l$-variable, where $l=1,\ldots,k$, and thus belongs to $\ker Q_U^{k}$. Therefore
	\[
	Q_U^{k}\circ d(\jmath_U(\pi_0|_U,\ldots,\pi_{k-1}|_U))=Q_U^{k+1}\circ\jmath_U(1,\pi_0,\ldots,\pi_{k-1}).
	\]
	It follows that $\bar d_U(\bar\jmath(\pi_0,\ldots,\pi_{k-1}))=\bar\jmath_U(1,\pi_0,\ldots,\pi_{k-1})$. Since $U$ is arbitrary we have
	\[
	\bar d(\bar\iota(\pi_0,\ldots,\pi_{k-1}))=\bar\iota(1,\pi_0,\ldots,\pi_{k-1})=\bar\iota(\varphi,\pi_0,\ldots,\pi_{k-1})
	\]
	for any $\varphi\in\LIP_c(X)$ which is 1 on a neighborhood of $\spt\pi_0$.
	This implies that
	\[
	\widehat T\circ\bar d(\bar\iota(\pi_0,\ldots,\pi_{k}))=\widehat T(\bar\iota(\varphi,\pi_0,\ldots,\pi_k))=T(\varphi,\pi_0,\ldots,\pi_k)=\partial T(\pi_0,\ldots,\pi_k).
	\]
The claim follows.
\end{proof}

\subsection{Pre-dual for metric currents -- Proof of Theorem \ref{thm:pre_dual}}
Before proving Theorem \ref{thm:pre_dual}, we record he following locality property of $\bar\iota$. Note that the claim of Lemma \ref{rmk:locality} is not true for $\iota$, and thus $\Gamma_c^k(X)$ is not a pre-dual for metric currents.
\begin{lemma}\label{rmk:locality}
	Let $(\pi_0,\ldots,\pi_k)\in\sD^k(X)$. If $\pi_l$ is constant on a neighborhood of $\spt\pi_0$ for some $l=1,\ldots,k$ then $\bar\iota(\pi_0,\ldots,\pi_k)=0$.
\end{lemma}
\begin{proof}
	If $x\in \spt\pi_0$ and $U$ is a neighborhood of $x$ so that $\pi_l$ is constant on $U$, then $\overline\Poly^k(U)\ni\pi_0|_U\otimes\overline{\pi_1|_U}\otimes\cdots\otimes\overline{\pi_k|_U}=0$. Thus $\bar\iota(\pi_0,\ldots,\pi_k)=[\pi_0\otimes\bar\pi_1\otimes\cdots\otimes\bar\pi_k]_x=0$
	
	If $x\notin\spt\pi_0$, then there is a neighborhood $V$ of $x$ on which $\pi_0$ vanishes, so that $\Poly^k(V)\ni\pi_0|_V\otimes\cdots\otimes\pi_k|_V=0$. 
\end{proof}
\begin{proof}[Proof of Theorem \ref{thm:pre_dual}]
By Theorem \ref{ext1} we have a linear map $$\Xi:\sD_k(X)\to \overline\Gamma_c^k(X)^\ast,\quad T\mapsto \widehat T.$$ The uniqueness of the extension implies injectivity $\Xi$. Indeed, the only current $T\in\sD_k(X)$ for which  $\widehat T=\Xi(T)=0$ is the zero current $T=0$.

If $\widehat T\in \overline\Gamma^k_c(X)^*$, then $T:=\widehat T\circ\bar\imath:\sD^k(X)\to \R$ defines a metric $k$-current. Indeed, $(k+1)$-linearity and sequential continuity are clear, and locality follows from Lemma \ref{rmk:locality}. Clearly $\Xi(T)=\widehat T$ and thus we have shown the surjectivity of $\Xi$.

The identity (\ref{main1:partial}) is proven in Proposition \ref{prop:bdry}. It remains to prove the sequential continuity of $\Xi$. By linearity, it suffices to show that, if $T_i\to 0$ in $\sD_k(X)$, then $\widehat T_i\to 0$ in $\Gamma_c^k(X)^*$.

Let $\omega\in \overline\Gamma_c^k(X)$ and suppose $\{ \bar\pi_U \}_{U\in\mathcal U}$ is overlap-compatible with $\omega$. Let $\{ \varphi_U \}_{U\in \mathcal U}$ be a Lipschitz partition of unity subordinate to $\mathcal U$. Since $\omega$ is compactly supported, there is a compact set  $K\subset X$ containing every set $U\in \mathcal U$ for which $\bar\pi_U\ne 0$. The collection of these elements of $\mathcal U$ is a finite set and we denote it by $\mathcal U_K$.

For each $U\in\mathcal U_K$ let $\pi_U\in Q_U\inv(\bar\pi_U)$ and fix a representation $(\pi_0^{j,U},\ldots,\pi_k^{j,U})_j\in\Rep(\pi_U)$. We have
\begin{align*}
\widehat T_i(\omega)=\sum_{U\in\mathcal U_K}\overline{(T_i)_{\varphi_U}^U}(\bar\pi_U)=\sum_{U\in\mathcal U_K}\sum_j^\infty T_i(\varphi_U\pi_0^{j,U},\pi_1^{j,U},\ldots,\pi_k^{j,U}).
\end{align*}
Since $\lim_{i\to\infty}(T_i)_{\varphi_U}^U(\pi_0,\ldots,\pi_k)=0$ for every $(\pi_0,\ldots,\pi_k)\in\LIP_\infty(U)^{k+1}$ the bound (\ref{bound}) and the multilinear uniform boundedness principle, cf. \cite[Theorem 1]{san85} and \cite[Theorem 1]{thi09}, implies that there is a constant $C>0$ for which
\begin{equation*}\label{eq:1}
|T_i(\varphi_U\pi_0,\pi_1,\ldots,\pi_k)|\le CL(\pi_0)\cdots L(\pi_k)
\end{equation*}
for all $i\in\N$ and every $(\pi_0,\ldots,\pi_k)\in\LIP_\infty(U)^{k+1}$. By the dominated convergence theorem, we have that
\[ 
\lim_{i\to\infty}\widehat T_i(\omega)=\sum_{U\in\mathcal V}\sum_j^\infty\lim_{i\to\infty}T_i(\varphi_U\pi_0^{j,U},\pi_1^{j,U},\ldots,\pi_k^{j,U})=0. 
\]
The proof is complete.
\end{proof}


\section{Currents of locally finite mass: extension to partition-continuous polylipschitz forms} 
\label{sec:part_cont}

In this section we show that currents of \emph{locally finite mass} admit a further extension to \emph{partition-continuous} sections. We introduce the following notation: for a subset $E\subset X$ and $\bar\pi\in\overline\Poly^k(X)$, set
\[
\overline L_k(\bar\pi;E):=\overline L_k(\rho_{E,X}(\bar\pi))
\]
and
\[\begin{split}
\Lip_k(\bar\pi; E)=\inf\left\{ \sup_{x\in E}\sum_j^\infty |\pi_0^j(x)|\Lip(\pi_1^j|_E)\cdots \Lip(\pi_1^j|_E): (\pi_0^j,\ldots,\pi_k^j)\in\Rep(\pi) \right\}
\end{split}\]
for any $\pi\in Q\inv(\bar\pi)$. This is clearly independent of the choice of $\pi$ and the estimate
\[ \Lip_k(\bar\pi;E)\le \overline L_k(\bar \pi;E) \] holds.

\begin{definition}\label{pwnorm}
	For $\omega\in \overline\G^k(X)$ and $x\in X$, the \emph{pointwise norm} $\|\omega\|_x$ of $\omega$ at $x$ is \[ \|\omega\|_x=\lim_{r\to 0}\Lip_k(\bar\pi;B_r(x)), \] where  $\bar\pi\in \overline\Poly^k(U)$ satisfies $[\bar\pi]_x=\omega(x)$ and $U$ is a neighborhood of $x$.
\end{definition}

A simple argument using a representation of $\omega$ shows that, for each $\omega \in \Gamma^k(X)$, the function $x\mapsto \|\omega\|_x$ is upper semicontinuous. 

We have, for $\pi=(\pi_0,\ldots,\pi_k)\in \sD^k(X)$, that
\[ 
\|\bar\iota(\pi)\|_x\le |\pi_0(x)|\Lip\pi_1(x)\cdots\Lip\pi_k(x) 
\] 
for all $x\in X$.

\subsection{Partition continuous polylipschitz forms and their convergence}
Let $E\subset X$ be a Borel set, and define
\[
\overline\G^k_E(X):=\rho_E\inv(\overline\Gamma^k(E)).
\]
Given $\omega\in\overline\G^k_E(X)$, a collection $\{\bar\pi_U\}_\cU$ is said to be \emph{overlap-compatible with $\omega$ in $E$}, if $\cU$ is a locally finite precompact open cover of $E$ and $\{ \bar\rho_{E\cap U,U}(\bar\pi_U) \}_\cU$ is overlap-compatible with $\omega|_E$. Set 
\[
\overline\G^k_{E,c}(X)=\overline\G^k_E(X)\cap \overline\G^k_c(X).
\]

\begin{definition}\label{def:pcpolylip}
Let $\mathcal E$ be a countable Borel partition of $X$. A section $\omega\in\overline\G^k(X)$ is called $\mathcal E$-continuous if there is a locally finite precompact open cover $\cU$ of $X$ and a collection $\{ \bar\pi_U\in\overline\Poly^k(U) \}_\cU$ so that
\begin{itemize}
	\item[(1)] $\{\bar\pi_U\}_\cU$ is overlap-compatible with $\omega$ in $E$, for every $E\in\mathcal E$;
	\item[(2)] $\displaystyle \sup_{E\in\mathcal E}\overline L_k(\pi_U;E\cap U)=C_U<\infty$ for every $U\in\cU$.
\end{itemize}
Given $\mathcal E$ and $\{ \bar\pi_U \}_\mathcal U$ satisfying (1) and (2) above we say that $\{\pi_U\}_\cU$ represents $\omega$ with respect to $\mathcal E$.

A section $\omega\in\overline\G^k(X)$ is called \emph{partition-continuous} if it is $\mathcal E$-continuous for some countable Borel partition of $X$. We also call a partition $\mathcal E$, for which $\omega$ is $\mathcal E$-continuous, an \emph{admissible partition for $\omega\in \overline\G^k(X)$.}
\end{definition}
We denote $\overline\Gamma^k_{\operatorname{pc}}(X)$ the vector space of of partition-continuous forms, and set $\hParcon(X):=\overline\G_c^k(X)\cap \overline\Gamma^k_{\operatorname{pc}}(X)$.

\begin{definition}\label{def:pcpolylipconv}
A sequence $(\omega_n)$ in $\hParcon(X)$ converges to $\omega\in\hParcon(X)$, denoted $\omega_n\to \omega$ in $\hParcon(X)$, if there is a compact set $K\subset X$, a countable Borel partition $\mathcal E$, a locally finite precompact open cover $\mathcal U$ and collections $\{ \bar\pi_U^n \}_\mathcal U$ representing $\omega_n-\omega$ with respect to $\mathcal E$, $n\in\N$, so that
\begin{itemize}
	\item[(1)] $\spt(\omega_n-\omega)\subset K$, for each $n\in\N$,
	\item[(2)] $\rho_{E\cap U,U}(\bar\pi^n)\to 0$ in $\overline\Poly^k(E\cap U)$ for all $E\in\mathcal E$ and $U\in\cU$, and 
	\item[(3)] $\displaystyle \sup_{n,E} \overline L_k(\bar\pi^n;E\cap U)=C_K<\infty$ for each $U\in \cU_K:=\{U\in\cU: U\cap K\ne\varnothing \}$.
\end{itemize}
\end{definition}

Note that the inclusion $\imath_{pc}:\overline\Gamma_c^k(X)\hookrightarrow \hParcon(X)$ is sequentially continuous. We denote by
\[
\iota_{pc}:=\imath_{pc}\circ\iota:\sD^k(X)\to \hParcon(X)
\]
the natural inclusion.

\subsection{Mass bounds for restrictions of currents to Borel sets} In this section we prove Theorem \ref{main2}. In fact the bound \eqref{main1:mass} in Theorem \ref{main2} follows directly from the more technical statement \eqref{eq:intbound} in Proposition \ref{prop:eext}. We begin by discussing a variant of Lemma \ref{lem:need} for currents of locally finite mass and their restrictions to Borel sets.

Let $T\in M_{k,\loc}(X)$ and $E\subset X$ be a Borel set. If $\varphi\in\LIP_\infty(E)$ is the restriction to $E$ of a compactly supported function (equivalently, if $\spt\varphi$ is a totally bounded set) and $U\subset X$ is an open set with $\spt\varphi\subset U$, consider the map $(T\lfloor E)_\varphi^U$ in (\ref{eq:map}). By the locality properties of currents the value $(T\lfloor E)_\varphi^U(\pi_0,\ldots,\pi_k)$ for $(\pi_0,\ldots,\pi_k)\in\LIP_\infty(U)^{k+1}$ depends only on $\pi_0|_{E\cap U},\ldots,\pi_k|_{E\cap U}$. Thus we get a $(k+1)$-linear map
\[
T_\varphi^{E\cap U}:\LIP_\infty(E\cap U)^{k+1}\to\R.
\]
By Lemma \ref{moi}, the map $T_\varphi^{E\cap U}$ satisfies the bound
\begin{equation}\label{eq:pcbound}
|T_\varphi^{E\cap U}(\pi_0,\ldots,\pi_k)|\le \Lip(\pi_1|_{E\cap U})\cdots\Lip(\pi_k|_{E\cap U})\int_{E\cap U}|\varphi||\pi_0|\ud\|T\|.
\end{equation}
In particular it is bounded and satisfies (\ref{eq:locality}). We denote by
\begin{equation*}
\overline{T_\varphi^{E\cap U}}:\overline\Poly^k(E\cap U)\to\R
\end{equation*}
the sequentially continuous linear map for which $T_\varphi^{E\cap U}=\overline{T_\varphi^{E\cap U}}\circ\bar\jmath_{E\cap U}$; cf. Lemma \ref{lem:need}. Note that the statements in Remark \ref{rmk:functorial} remain true for all Borel sets $U\subset V\subset X$.

Given $T\in M_{k,\loc}(X)$ and a Borel set $E\subset X$, we define $\widehat T_E:\G^k_{E,c}(X)\to\R$ by
\begin{equation}\label{eq:E}
\widehat T_E(\omega):= \sum_{U\in\mathcal U}\overline{T_{\varphi_U}^{E\cap U}}(\bar\rho_{E\cap U,U}(\bar\pi_U)),
\end{equation}
whenever $\omega\in \overline\G_{E,c}^k(X)$ is represented by $\{\pi_U\}_\cU$ in $E$, and $\{\varphi_U\}$ is a Lipschitz partition of unity (in $E$) subordinate to $\cU$.
\begin{proposition}\label{prop:eext}
	Let $T\in M_{k,\loc}(X)$ and let $E\subset X$ be a Borel set. Then $\widehat T_E$ is well-defined, linear and satisfies $T\lfloor_E=\widehat T_E\circ\bar\iota$. Moreover,
	\begin{equation}\label{eq:intbound}
	|\widehat T_E(\omega)|\le \int_E\|\omega|_E\|_x\ud\|T\|(x),\quad \textrm{ for each }\omega\in\overline\G^k_{E,c}(X).
	\end{equation}
\end{proposition}
\begin{proof}
	Arguing as in the proof of Theorem \ref{ext1}, we see that $\widehat T_E$ is well-defined. Linearity is clear and the factorization follows from the identities $T_{\varphi_U}^{E\cap U}=\overline{T_{\varphi_U}^{E\cap U}}\circ\bar\jmath_{E\cap U}$ as in the proof of Theorem \ref{ext1}.
	
	Next we prove the bound in the claim. Let $\varphi\in\Lip_\infty(E)$ and $\spt \varphi\subset U$ a precompact open set. For $\pi\in \overline \Poly^k(X)$,  $\pi\in Q\inv(\bar\pi)$ and  $(\pi_0^j,\ldots,\pi_k^j)_j\in\Rep(\pi)$, we have
		\[
		\overline{T_\varphi^{E\cap U}}(\bar\pi)=\sum_j^\infty T_\varphi^{E\cap U}(\pi_0^j,\ldots,\pi_k^j)=\sum_j^\infty T(\chi_E\varphi\pi_0^j,\pi_1^k,\ldots,\pi_k^j)
		\]
		and, by (\ref{eq:pcbound}), the estimate 
		\begin{align*}
		|\overline{T_\varphi^{E\cap U}}(\bar\pi)| &\le \sum_j^\infty\Lip(\pi_1^j|_{E\cap\spt\varphi})\cdots\Lip(\pi_k^j|_{E\cap\spt\varphi})\int_E|\varphi\pi_0^j|\ud\|T\| \\
		&\le \sup_{x\in E\cap \spt\varphi}\left(\sum_j^\infty\Lip(\pi_1^j|_{E\cap\spt\varphi})\cdots\Lip(\pi_k^j|_{E\cap\spt\varphi})|\pi_0^j(x)|\right)\int_E|\varphi|\ud\|T\|,
		\end{align*}
		yielding 
		\begin{equation}\label{moii}
		|\overline{T_\varphi^{E\cap U}}(\bar\pi)|\le \Lip_k(\bar\pi;E\cap \spt\varphi)\int_E|\varphi|\ud\|T\|.
		\end{equation}
	
	Let $\omega\in \overline\G_{E,c}^k(X)$. We extend $\omega|_E:E\to\overline\Polys^k(E)$ as the zero map
	$$
	\omega|_E:X\to \overline\Polys^k(E)\sqcup\overline\Polys^k(X\setminus E)
	$$
	outside $E$, and thus $\|\omega|_E\|_x=0$ if $x\notin E$. Let $K:=\spt\omega$ and let $g:X\to\R$ be a simple Borel function satisfying $\|\omega|_E\|\le g.$ We may assume that 
	\[
	g=\sum_l^ma_l\chi_{A_l},
	\]
	where $m\in\N$ and $\{A_0,\ldots, A_m\}$ is a Borel partition of $K$.
	
	Let $\varepsilon>0$. Since $\|T\|$ is a Radon measure there is, for each $l$, an open set $U_l\supset A_l$, satisfying 
	\[
	\|T\|(U_l)<\|T\|(A_l)+\varepsilon.
	\]
	We construct a collection overlap-compatible with $\omega$ in $E$ as follows: for $x\in K$ there exists a unique $l=0,\ldots,m$ for which $x\in A_l$. Fix a radius $r_x>0$ for which $B_x:=B_{r_x}(x)\subset U_l$ is precompact. Let $\pi_x\in\overline\Poly^k(B_x)$ satisfy 
	\[
	[\bar\rho_{E\cap B_x,B_x}(\pi_x)]_x=\omega|_E(x)\textrm{ and }\Lip_k(\bar\pi_{x}; E\cap B_x)\le a_l+\varepsilon,
	\]
	if $x\in E$, and $\pi_x=0$ if $x\in K\setminus E$. For $x\notin K$ let $B_x\subset X\setminus K$ be a precompact ball around $x$ and set $\pi_x=0$.
	
	By Lemma \ref{lem:overlap}, we may pass to a locally finite refinement $\cU$ and a collection
	\[
	\pi_U:=\rho_{U,B_x}(\pi_x)\in \overline\Poly^k(U), \textrm{ where }U\subset  B_x
	\]
	so that $\{ \bar\rho_{E\cap U, U}(\bar\pi_U) \}_\cU$ is overlap-compatible with $\omega|_E$. Note that $\rho_{E\cap U,U}(\bar\pi_U)=0$ if $U\notin\cU_K:= \{ U\in \cU: U\cap K\ne\varnothing \}$, since otherwise there would be $x\in E\setminus K$ for which $\omega(x)=[\rho_{E\cap U,U}(\bar\pi_U)]_x\ne 0$.

	Let $\{\varphi_U\}$ be a Lipschitz partition of unity subordinate to $\cU$ (in $E$). By the definition of $\widehat T_E$ and (\ref{moii}) we have
	\begin{align*}
	|\widehat T_E(\omega)| &\le\sum_{U\in\cU_K}|\overline{T_{\varphi_U}^{E\cap U}}(\bar\pi_U)|\le \sum_{U\in\cU_K}\Lip_k(\pi_U;E\cap U)\int_E\varphi_U\ud\|T\|.
	\end{align*}
	We may express the collection $\cU_K$ as \[ \cU_K=\bigcup_{l=0}^m\cU_l,\quad\textrm{ where }\quad \cU_l=\{ U\in \cU: U\subset B_x \textrm{ with }x\in A_l \}. \]
	Thus 
	\begin{align*}
	|\widehat T_E(\omega)|\le & \sum_l^m\sum_{U\in\cU_l}\Lip_k(\pi_j;E\cap U)\int_E\varphi_j\ud\|T\|\\
	\le & \sum_l^m(a_l+\varepsilon)\int_E\left(\sum_{U\in\cU_l}\varphi_U\right)\ud\|T\|\\
	\le &\sum_l^m(a_l+\varepsilon)\int_E\chi_{U_l}\ud\|T\|\le \sum_l^m (a_l+\varepsilon)\|T\|(U_l\cap E).
	\end{align*}
	By the choice of the open sets $U_l$, we have that
	\begin{align*}
	\|T\|(U_l\cap E)+\|T\|(U_l\setminus E) &<\|T\|(A_l\cap E)+\|T\|(A_l\setminus E)+\varepsilon \\
	&< \|T\|(A_l\cap E)+\|T\|(U_l\setminus E)+\varepsilon.
	\end{align*}
	Thus $$\|T\|(U_l\cap E)<\|T\|(A_l\cap E)+\varepsilon$$ for each $l\le m$. Since $\varepsilon>0$ is arbitrary, we have 
	\[ 
	|\widehat T_E(\omega)|\le \sum_l^ma_l\|T\|(A_l\cap E)=\int_Eg\ud\|T\|. 
	\] 
	By taking infimum over all simple functions $g$ satisfying $\|\omega|_E\|\le g$, we obtain the claim.
\end{proof}

\begin{proposition}\label{prop:eextunique}
	Let $T\in \sD_k(X)$. The map $\widehat T_E:\overline\G_{E,c}^k(X)\to \R$ in \eqref{eq:E} is unique among linear maps satisfying \eqref{eq:intbound}.
\end{proposition}
\begin{proof}
	Let $A:\overline\G_{E,c}^k(X)\to\R$ be a linear map such that $A\circ\bar\iota=T\lfloor E$ and $A$ satisfies (\ref{eq:intbound}). We observe that,  by (\ref{eq:intbound}), the value $A(\omega)$ depends only on $\omega|_E$.	
	
	Let $\omega\in\overline\Gamma_{E,c}^k(X)$ and suppose $\{ \bar\pi_U \}_\cU$ is overlap-compatible with $\omega$ in $E$, and let $\{\varphi_U \}$ be a Lipschitz partition of unity in $E$ subordinate to $\cU$. For each $U\in\cU$ note that $\varphi_U\smile\omega|_E=\bar\gamma_E(\varphi_U\smile\bar\rho_{E\cap U,U}(\bar\pi_U))$. Consider the multilinear map 
	\[
	A_U:\LIP_\infty(U\cap E)^{k+1}\to\R,\quad A_U= A(\bar\gamma_E(\varphi_U\smile\jmath_{E\cap U}(\cdot)))
	\]
	where $\bar\gamma_E:\Poly^k(E)\to \Gamma^k(E)$ is the canonical map in (\ref{eq:gammaembed}). By (\ref{eq:intbound}), $A_U$ is bounded. As in the proof of Theorem \ref{ext1} we see that $A_U=T_{\varphi_U}^{E\cap U}$. Thus
	\[
	\overline{T_{\varphi_U}^{E\cap U}}=\overline{A_U}=A(\bar\gamma_E(\varphi_U\smile\bar\rho_{E\cap U,U}(\cdot))):\Poly^k(E\cap U)\to \R
	\]
	It follows that
	\begin{align*}
	A(\omega)&=\sum_{U\in\cU}A(\varphi_U\smile\omega)=\sum_{U\in\cU}A(\bar\gamma_E(\varphi_U\smile\bar\rho_{E\cap U,U}(\bar\pi_U)))\\
	&=\sum_{U\in\cU}\overline{T_{\varphi_U}^{E\cap U}}(\bar\pi_U)=\widehat T_E(\omega).
	\end{align*}
\end{proof}

\subsection{Extending currents of locally finite mass} The remainder of this section is devoted to the proof of Theorem \ref{ext2}. The existence and uniqueness of the extension is proved in Proposition \ref{prop:pcext} below.

\begin{proposition}\label{prop:pcext}
Let $T\in M_{k,\loc}(X)$ and let $\widehat T:\hParcon(X)\to\R$ be the linear map
\[
\omega \mapsto \sum_{E\in\mathcal E}\widehat T_E(\omega)
\]
whenever $\mathcal E$ is an admissible partition for $\omega$. Then $\widehat T$ is well-defined. Moreover, $\widehat T$ is the unique sequentially continuous linear map $\overline\Gamma_{pc,c}^k(X)\to \R$ satisfying
\[
\widehat T\circ\bar\iota_{pc}=T.
\]
\end{proposition}
\begin{proof}
 If $\varphi\in\LIP_c(X)$, $\spt\varphi\subset U$ is open, and $E_1,E_2\subset X$ are disjoint Borel sets  $E=E_1\cup E_2$, the identity $T\lfloor E=T\lfloor E_1+T\lfloor E_2$ implies that
\[
\overline{T^{E\cap U}_\varphi}=\overline{T^{E_1\cap U}_\varphi}\circ\bar\rho_{E_1\cap U,E\cap U}+\overline{T^{E_2\cap U}_\varphi}\circ\bar\rho_{E_2\cap U,E\cap U}.
\]
This and the estimate (\ref{moii}) can be used to show that, if $E_1,E_2,\ldots$ is a Borel partition of $E$, we have
\[
\widehat T_E(\omega)=\sum_i^\infty \widehat T_{E_i}(\omega),\quad \omega\in\G^k_{E,c}(X).
\]
Note that by (\ref{eq:intbound}) the sum above is absolutely convergent. The well-definedness of $\widehat T:\hParcon(X)\to \R$ follows easily from this.

The factorization $\widehat T\circ\iota_{pc}=T$ follows immediately from the observation that $\widehat T|_{\overline\Gamma_c^k(X)}=\widetilde T$, where $\widetilde T$ denotes the extension given by Theorem \ref{ext1}. Uniqueness follows from Proposition \ref{prop:eextunique} and (\ref{eq:intbound}).

Next we prove that $\widehat T$ is sequentially continuous. Suppose $\omega_n\to \omega$ in $\hParcon(X)$, and let $K\subset X$, $\mathcal E$ and $\{ \bar\pi^n \}_\cU$ be as in Definition \ref{def:pcpolylipconv}. Denote $\cU_K:=\{ U\in\cU:\ U\cap K\ne\varnothing \}$  and set 
\[
C_K:=\sup_{n,E,U\in\cU_K}\bar L_k(\pi^n;E\cap U)<\infty.
\]
Let $\{\varphi_U\}$ be a Lipschitz partition of unity subordinate to $\cU$. For each $n\in\N$, we have
\begin{align*}
\widehat T(\omega_n)-\widehat T(\omega)=\widehat T(\omega_n-\omega)=\sum_{E\in\mathcal E}\sum_{U\in\cU}\overline{T^{E\cap U}_{\varphi_U}}(\bar\pi^n)=\sum_{E\in\mathcal E}\sum_{U\in\cU_K}\overline{T^{E\cap U}_{\varphi_U}}(\bar\pi^n).
\end{align*}
Since 
\[
|\overline{T^{E\cap U}_{\varphi_U}}(\bar\pi^n)|\le \Lip_k(\bar\pi^n;E\cap U)\int_E\varphi_U\ud\|T\|\le C_K\int_E\varphi_U\ud\|T\|
\]
we may apply the dominated convergence theorem to conclude that
\[
\lim_{n\to\infty}\widehat T(\omega_n-\omega)=\sum_{E\in\mathcal E}\sum_{U\in\cU_K}\lim_{n\to\infty}\overline{T^{E\cap U}_{\varphi_U}}(\bar\pi^n)=0.
\]
The claim follows.
\end{proof}

\subsection{Boundaries of normal currents} We finish the proof of Theorem \ref{ext2} by showing the validity of \eqref{eq:pcd} in Corollary \ref{pcbdry}.

\begin{proposition}\label{prop:pcdext}
The differentials $d:\G^k_c(X)\to\G^{k+1}_c(X)$ and $\bar d:\overline\G^k_c(X)\to\overline\G^{k+1}_c(X)$ restrict to sequentially continuous linear maps
\[
d:\Parcon(X)\to \Gamma^{k+1}_{\mathrm{pc},c}(X) \textrm{ and }\bar d:\hParcon(X)\to \overline\Gamma^{k+1}_{\mathrm{pc},c}(X).
\]
\end{proposition}
\begin{proof}
	We prove the statement for $\bar d$. The other case is similar. Let $\omega\in \hParcon(X)$ and let $\mathcal E$, $\{\bar\pi_U\}_\cU$, and $C_U$ be as in Definition \ref{def:pcpolylip}. Since $\bar d(\omega|_E)=(\bar d\omega)|_E$, Remark \ref{rmk:homomcomp} implies that $\{ \bar d\bar\pi_U\}_\cU$ is overlap-compatible with $\bar d\omega$ in $E$ for every $E\in\mathcal E$. Moreover by Lemma \ref{extdcont} we have
	\[
	\bar L_{k+1}(\bar d\bar\pi_U; E\cap U)\le C_U(k+2)
	\]
	for every $E\in\mathcal E$. Thus $\bar d\omega\in\overline\Gamma_{\operatorname{pc},c}^{k+1}(X)$.
	
	To see sequential continuity, let $\omega_n\to \omega$ in $\hParcon(X)$, and let $\mathcal E$, $K$, $C_K$, and $\{ \bar\pi^n \}_\cU$ be as in Definition \ref{def:pcpolylipconv}. Since $\rho_{E\cap U,U}(\bar\pi^n)\to 0$ in $\overline\Poly^k(E\cap U)$ and $\bar d_{E\cap U}$ is sequentially continuous it follows that $\bar d_{E\cap U}(\rho_{E\cap U,U}(\bar\pi^n))\to 0$ in $\overline\Poly^{k+1}(E\cap U)$. Thus $\bar d(\omega_n-\omega)\to 0$ in $\overline\Gamma_{\operatorname{pc},c}^{k+1}(X)$ and the proof of the proposition is complete.
\end{proof}

Proposition \ref{prop:pcdext} and the uniqueness in Proposition \ref{prop:pcext} 
immediately yield the following corollary.

\begin{corollary}\label{pcbdry}
	Let $T\in N_{k,\loc}(X)$ be a locally normal $k$-current on $X$. Let $\widehat T:\Parcon(X)\to\R$ and $\widehat{\partial T}:\Gamma_{\operatorname{pc},c}^{k-1}(X)\to\R$ be extensions of $T$ and $\partial T$, respectively. Then we have \[ \widehat{\partial T}(\omega)=\widehat T(d\omega)\] for each $\omega\in \Gamma_{\operatorname{pc},c}^{k-1}(X)$.
\end{corollary}

\begin{remark}
Corollary \ref{pcbdry} implies that $\partial \widehat T:\Gamma_{\operatorname{pc},c}^{k-1}(X)\to \R$, $\omega\mapsto \widehat T(d\omega)$, coincides with the extension $\widehat{\partial T}:\Gamma_{\operatorname{pc},c}^{k-1}(X)\to\R$ of $\partial T\colon \sD^k(X)\to\R$ to partition-continuous polylipschitz forms. Thus the use of the symbol $\partial$ is unambiguous.
\end{remark}

\section{Final remarks}\label{sec:Alt}
We briefly discuss polylipschitz sections in connection with duality, and antisymmetrization on polylipschitz forms.

\subsection{Extending currents to polylipschitz sections}

Define the space $\Parcon(X)$ and sequential convergence in $\Parcon(X)$ as in Definitions \ref{def:pcpolylip} and \ref{def:pcpolylipconv}. It is straightforward to check that the map $\mathcal Q:\G^k(X)\to \overline\G^k(X)$ in \eqref{eq:sectquotient} restricts to a sequentially continuous linear map
\[
\mathcal Q:\Parcon(X)\to\overline\Gamma_{pc,c}^k(X).
\]
We record the following theorem for extensions of currents to polylipschitz sections. The claims follow directly from Theorems \ref{ext2} and \ref{main2} together with the fact that $\bar d\circ\mathcal Q=\mathcal Q\circ d$, cf.~Proposition \ref{prop:hcohomom}.

\begin{theorem}\label{thm:polysec}
Suppose $T\in M_{k,\loc}(X)$, let $\widehat T$ be the unique extension given by Proposition \ref{prop:pcext}, and $\widetilde T:=\widehat T\circ \mathcal Q$. Then $\widetilde T:\Parcon(X)\to\R$ linear, sequentially continuous and satisfies $\widetilde T\circ\iota=T$.

Moreover, for $T\in M_{k,\loc}(X)$ and $T'\in N_{k+1,\loc}(X)$, the identities
\[
\widetilde{\partial T'}(\omega)=\widetilde T'(d\omega)
\]
and
\[
|\widetilde T(\omega)|\le \int_X\|\omega\|_x\ud\|T\|(x)
\]
hold for all $\omega\in \Parcon(X)$.
\end{theorem}

\subsection{Alternating polylipschitz forms and metric currents}

In \cite{amb00}, Ambrosio and Kirchheim point out that the other axioms of metric currents imply that a metric $k$-current $T$ on space $X$ is alternating in the sense that 
\[
T(\pi_0, \pi_{\sigma(1)},\ldots, \pi_{\sigma(k)}) = \sign(\sigma) T(\pi_0,\pi_1,\ldots, \pi_k)
\]
for all $\pi=(\pi_0,\ldots, \pi_k)\in \sD^k(X)$ and permutations $\sigma$ of $\{1,\ldots, k\}$. 

Taking into account the particular role of the function $\pi_0$ in the $(k+1)$-tuple $(\pi_0,\ldots, \pi_k)$ in the definition of a $k$-current, we use this property of metric currents to define an antisymmetrization operator $\Alt =\Alt_X\colon \Poly^k(X) \to \Poly^k(X)$ by
\[
\Alt(\pi)(x_0,\ldots, x_k) = \frac{1}{k!}\sum_{\sigma} \sign(\sigma) \pi(x_0, x_{\sigma(1)},\ldots, x_{\sigma(k)})
\]
for $\pi \in \Poly^k(X)$ and $(x_0,\ldots, x_k)\in X^{k+1}$. This map descends to a linear map
\[
\overline\Alt:\overline\Poly^k(X)\to \overline\Poly^k(X)\quad \textrm{ satisfying } Q\circ\Alt=\overline\Alt\circ Q.
\]
We call the images $\Alt(\Poly^k(X))$ and $\overline\Alt(\overline\Poly^k(X))$ \emph{alternating polylipschitz functions} and \emph{alternating homogeneous polylipschitz functions}, respectively.

Continuous sections of the \'etal\'e space associated to the corresponding presheaves gives rise to \emph{alternating} polylipschitz sections and forms, $\Gamma^k_{\Alt}(X)$ and $\overline\Gamma_{\Alt}^k(X)$, respectively. Since the exterior derivatives $d$ and $\bar d$ preserve the property of being alternating on (homogeneous) polylipschitz functions, they induce exterior derivatives
\[
d \colon \Gamma^k_{\Alt}(X)\to \Gamma^{k+1}_{\Alt}(X),\quad \bar d \colon \overline\Gamma^k_{\Alt}(X)\to \overline\Gamma^{k+1}_{\Alt}(X).
\]

Since classical differential $k$-forms on a manifold may be viewed either as sections of the $k$th exterior bundle or as sections of the bundle of alternating $k$-linear functions, we observe that alternating polylipschitz forms on a metric space are analogous to the latter.

It is now straightforward to check, using the observation of Ambrosio and Kirchheim, that for each metric $k$-current $T$, we have
\[
\widetilde T=\widetilde T\circ\Alt\quad\textrm{and}\quad \widehat T=\widehat T\circ\overline\Alt,
\]
where $\Alt$ and $\overline\Alt$ are the linear maps associated to the presheaf homomorphisms $\{\Alt_U \}$ and $\{ \overline\Alt_U \}$.

\appendix
\section{Cohomorphisms and their associated linear maps}\label{sec:cohomom}

In this appendix, we define cohomomorphisms between presheaves and describe a general construction yielding a linear map associated to a given cohomomorphism. 

Let $f:X\to Y$ be a continuous map between paracompact Hausdorff spaces and let $A=\{A(U);\rho^A_{U,V}\}_U$  and $B=\{B(U);\rho^B_{U,V}\}_U$ be presheaves on $X$ and $Y$, respectively. A collection
\[
\{\varphi_U:B(U)\to A(f\inv U)\}_U
\]
of linear maps for each open $U\subset Y$, satisfying
\begin{equation}\label{eq:presheafhomom}
\varphi_U\circ\rho^B_{U,V}=\rho^A_{f\inv U,f\inv V}\circ\varphi_V\quad\textrm{ whenever }U\subset V,
\end{equation}
is called an \emph{$f$-cohomomorphism of presheaves}; cf. \cite[Chapter I.4]{bredon97}. For $f=\id:X\to X$, condition (\ref{eq:presheafhomom}) becomes (\ref{eq:idhomom}) and thus $\id$-cohomomorphisms are simply presheaf homomorphisms.

An $f$-cohomomorphism $\varphi:B\to A$ between presheaves induces a natural linear map
\[
\varphi^*:\G(\mathcal B(Y))\to \G(\mathcal A(X)),
\]
\emph{the linear map (on sections) associated to $\varphi$}. Given a global section $\omega:Y\to \mathcal B(Y)$, the section $\varphi^*(\omega):X\to \mathcal A(X)$ is defined as follows: for $x\in X$,
\begin{equation}\label{eq:assoc}
\varphi^*(\omega)(x):=[\varphi_U(g_U)]_x,
\end{equation}
where $U$ is a neighborhood of $f(x)$ and $g_U\in B(U)$ satisfies $\omega(f(x))=[g_U]_{f(x)}$.

To see that $\varphi^*(\omega)(x)$ is well-defined, suppose that $\omega(f(x))=[g_U]_{f(x)}=[g'_V]_{f(x)}$, i.e., that there is a neighborhood $D\subset U\cap V$ of $f(x)$ for which
\[
\rho^B_{D,U}(g_U)=\rho^B_{D,V}(g'_V).
\]
By (\ref{eq:presheafhomom}) we have that
\begin{align*}
\rho^A_{f\inv D,f\inv U}(\varphi_U(g_U))=\varphi_D(\rho^B_{D,U}(g_U))=\varphi_D(\rho^B_{D,V}(g'_V))=\rho^A_{f\inv D,f\inv V}(\varphi_V(g'_V));
\end{align*}
in particular $[\varphi_U(g_U)]_x=[\varphi_V(g'_V)]_x$.
\begin{remark}\label{rmk:homomcomp}
	Let $\omega\in \G(\mathcal B(Y))$ be compatible with $\{ g_U \}_\cU$. Suppose $U_V$ satisfies $V=f\inv U_V$ for each $V\in f\inv\cU$. Then, by (\ref{eq:presheafhomom}) and the fact that $\{g_U\}_\mathcal U$ is compatible $\omega$, we have that the collection $\{ \varphi_V(g_{U_V}) \}_{f\inv\cU}$ is compatible with $\varphi^*\omega$.
	
	If $\omega\in \Gamma(\mathcal B(Y))$ and $\{g_U \}_\cU$ represents $\omega$, then $\{ \varphi_V(g_{U_V}) \}_{f\inv\cU}$ represents $\omega$, and if $\{g_U \}_\cU$ satisfies the overlap condition (\ref{overlap}) then $\{ \varphi_V(g_{U_V}) \}_{f\inv\cU}$ also satisfies the overlap condition.
\end{remark}

We collect some fundamental properties of linear maps associated to cohomorphisms in the next proposition. 
\begin{proposition}\label{prop:assocproperties}
	Let $f:X\to Y$ and $g:Y\to Z$ be continuous maps between paracompact Hausdorff spaces and let $A=\{A(U)\}_U$, $B=\{B(U)\}_U$ and $C=\{C(U)\}_U$ be presheaves on $X,Y$ and $Z$ respectively. Suppose $$\varphi=\{\varphi_U:B(U)\to A(f\inv U)\}_U,\quad \varphi'=\{\varphi'_U:B(U)\to A(f\inv U) \}_U$$ are $f$-cohomomorphisms, and $$\psi=\{\psi_U:C(U)\to B(g\inv U)\}_U$$ is an $g$-cohomomorphism.
	
	\begin{itemize}
		\item[(1)] For $\omega\in \G(\mathcal B(Y))$ we have \[ \spt(\varphi^*\omega)\subset f\inv(\spt(\omega)). \]
		\item[(2)] The associated linear map $\varphi^*:\G(\mathcal B(Y))\to \G(\mathcal A(X))$ satisfies
		\begin{equation*}
		\varphi^*(\Gamma(\mathcal B(Y)))\subset \Gamma(\mathcal A(X)).
		\end{equation*}
		\item[(3)] The collection $\{ \varphi_U+\varphi'_U:B(U)\to A(f\inv U) \}_U$ is an $f$-cohomomorphism and
		\[
		(\varphi+\varphi')^*=\varphi^*+\varphi'^*:\G(\mathcal B(Y))\to \G(\mathcal A(X)).
		\]
		\item[(4)] The collection $\{\varphi_{g\inv U}\circ\psi_U:C(U)\to A((g\circ f)\inv U)\}_U$ is an $(g\circ f)$-cohomomorphism and
		\[
		(\varphi\circ\psi)^*=\varphi^*\circ\psi^*:\G(\mathcal C(Z))\to \G(\mathcal A(X)).
		\]
	\end{itemize}
\end{proposition}
\begin{remark}\label{rmk:bilinear}
	Given presheaves 
	\[
	A=\{ A(U) \}_U\textrm{ on $X$ and }B_1=\{B_1(U) \},\ B_2=\{ B_2(U) \}\textrm{ on }Y,
	\]
	and bilinear maps $\{ \varphi_U;B_1(U)\times B_2(U)\to A(f\inv U) \}$ an analogous construction gives an associated bilinear map 
	\[
	\varphi^*:\G(\mathcal B_1(Y))\times\G(\mathcal B_2(Y))\to \G(\mathcal A(X)).
	\]
	The induced bi-linear map satisfies (2) and (3) and also 
\begin{itemize}
\item[(1')] For each $(\omega,\sigma) \in \G(\mathcal B_1(Y))\times \G(\mathcal B_2(Y))$, we have $$\spt(\varphi^*(\omega,\sigma))\subset f\inv(\spt\omega\cap \spt\sigma).$$
\end{itemize}
	We will need this only for the case $\id:X\to X$ in the construction of cup products. The details are similar as above and we omit them.
\end{remark}
\begin{proof}[Proof of Proposition \ref{prop:assocproperties}]
	The proofs are straightforward and we merely sketch them. 
	
If $\varphi^*\omega(x)\ne 0$ then, since $\varphi_U$ is linear, (\ref{eq:assoc}) implies that $\omega(f(x))=[g_U]_{f(x)}\ne 0$, proving (1). Claim (2) follows directly from Remark \ref{rmk:homomcomp}.
		
To prove (3) we observe that from (\ref{eq:assoc}) it is easy to see that, if $\varphi':B\to A$ is another $f$-cohomorphism between presheaves, then $\varphi+\varphi'$ is an $f$-cohomomorphism and we have
		\[
		(\varphi+\varphi')^*=\varphi^*+\varphi'^{*}.
		\]
		
To prove (4), note that condition (\ref{eq:presheafhomom}) follows for $\phi\circ\psi$ from the fact that it holds for $\varphi$ and $\psi$. Using (\ref{eq:assoc}) (and the same notation) we see that
		\[
		(\varphi\circ\psi)^*\omega(x)=[\varphi_{g\inv U}(\psi_U(g_U))]_x=\varphi^*(\psi^*\omega)(x).
		\]
\end{proof}

Proposition \ref{prop:assocproperties} has the following immediate corollary.

\begin{corollary}
	If $f:X\to Y$ is a proper continuous map and $\varphi:B\to A$ an $f$-cohomomorphism between presheaves $B$ on $Y$ and $A$ on $X$, then
	\[
	\varphi^*(\G_c(\mathcal B(Y)))\subset \G_c(\mathcal A(X))\textrm{ and }\varphi^*(\Gamma_c(\mathcal B(Y)))\subset \Gamma_c(\mathcal A(X)).
	\]
	In particular presheaf homomorphisms always have this property.
\end{corollary}

\bibliographystyle{plain}
\bibliography{abib}

\end{document}